\documentclass[reqno]{amsart}
\usepackage[margin=1.15in]{geometry}
\usepackage{amssymb}
\usepackage[backref=page]{hyperref}
\usepackage{cite}
\usepackage{enumitem}
\theoremstyle{plain}
\newtheorem{theorem}{Theorem}
\newtheorem{proposition}[theorem]{Proposition}
\newtheorem{corollary}[theorem]{Corollary}
\newtheorem{lemma}[theorem]{Lemma}
\theoremstyle{definition}
\newtheorem{definition}[theorem]{Definition}
\theoremstyle{remark}
\newtheorem{remark}[theorem]{Remark}
\allowdisplaybreaks

\newcommand{\R}{\mathbb{R}}
\newcommand{\N}{\mathbb{N}}
\newcommand{\eps}{\varepsilon}
\newcommand{\dist}{\operatorname{dist}}
\newcommand{\supp}{\operatorname{supp}}
\newcommand{\Lap}[1]{\Delta_{#1}}

\begin{document}
% authors
\author{Phuong Le}
\address{Phuong Le$^{1,2}$ (ORCID: 0000-0003-4724-7118)\newline
	$^1$Faculty of Economic Mathematics, University of Economics and Law, Ho Chi Minh City, Vietnam; \newline
	$^2$Vietnam National University, Ho Chi Minh City, Vietnam}
\email{phuongl@uel.edu.vn}
% classifications
\subjclass[2020]{35B51, 35J92, 35B65, 35B50, 35J70}
\keywords{$p$-Laplace equations; strong comparison principle; weak Harnack inequality;
weighted Sobolev inequalities; Morrey estimates; critical set; Moser iteration}
% other infos
%\date{March 13, 2021}
%\dedicatory{Dedicated to our adviser}
%\thanks{}
% title & abstract
\title[An improved strong comparison principle]{An improved strong comparison principle for
singular $p$-Laplace equations, with applications}

\begin{abstract}
We consider positive weak solutions of $-\Delta_p u=f(u)$, with $f$ positive and locally Lipschitz
continuous. In the singular case $1<p<2$, the Harnack-type inequalities, the strong maximum
principle for the linearized operator, the description of the critical set and the strong comparison
principle established by Damascelli and Sciunzi [Calc.\ Var.\ Partial Differential Equations
\textbf{25} (2006), 139--159] have been available, for more than twenty years, only under the
restriction $\frac{2N+2}{N+2}<p<2$, a range that shrinks to the empty set as $N\to\infty$. We remove
the dependence on the dimension and prove all of these results for $\frac32<p<2$.

The new ingredient is a uniform Morrey estimate for negative powers of the gradient: for every
$0\le\sigma<p-1$ there is $\kappa>0$ with
$\int_{B_r(x_0)}|Du|^{-\sigma}dy\le C\,r^{N-2+\kappa}$ for \emph{all} centres $x_0$ and all small
$r$, whereas the Riesz-potential bounds in the literature stop at $N-2$.
Combined with a trace inequality of Adams' type it yields a Sobolev inequality carrying the singular
weight $\rho=|Du|^{p-2}$ on \emph{both} sides, so that Moser's iteration can be run with respect to
$\rho\,dx$ rather than the Lebesgue measure. This replaces the requirement $\rho\in L^{t}$, $t>N/2$,
exactly what forced the exponent $\frac{2N+2}{N+2}$, by the local integrability of $\rho$, which
holds precisely for $p>\frac32$. We also show that $\frac32$ is optimal for
\emph{any} argument based on the weighted space $H^{1,2}_{\rho}$: already for $-\Delta_pu=1$ on an
annulus, with zero Dirichlet data, the critical set is a smooth hypersurface and $\rho\notin L^1$
whenever $1<p\le\frac32$; and we isolate an abstract Morrey criterion which, for instance,
gives the strong comparison principle in a ball for every $p>\frac43$.

Since the strong comparison principle is used as a black box throughout the qualitative theory of
$p$-Laplace equations, the improvement propagates. We show in particular that the assumption $\frac{2N+2}{N+2}<p<2$ may
be replaced by $\frac32<p<2$ in the resolution of Gibbons' conjecture for $-\Delta_pu=f(u)$ by
Esposito, Farina, Montoro and Sciunzi [Math.\ Ann.\ \textbf{382} (2022), 943--974] and in their
monotonicity theorem in half-spaces for changing-sign nonlinearities [Calc.\ Var.\ Partial
Differential Equations \textbf{61} (2022), art.\ 154]; further applications are given in the body of
the paper.
\end{abstract}
\maketitle
%\tableofcontents

%%%%%%%%%%%%%%%%%%%%%%%%%%%%%%%%%%%%%%%%%%%%%%%%%%%%%%%%%%%%%%%%%%%%%%%%%%%%%%%
\section{Introduction and statement of the results}
%%%%%%%%%%%%%%%%%%%%%%%%%%%%%%%%%%%%%%%%%%%%%%%%%%%%%%%%%%%%%%%%%%%%%%%%%%%%%%%

Let $\Omega\subseteq\R^N$, $N\ge2$, be a domain and let $1<p<\infty$. We consider weak solutions
$u\in C^1(\Omega)$ of
\begin{equation}\label{eq:main}
-\Lap{p}(u)=f(u),\qquad u>0\quad\text{in }\Omega,
\end{equation}
where $\Lap{p}(u)=\operatorname{div}(|Du|^{p-2}Du)$ is the $p$-Laplace operator and $f$ satisfies
\begin{itemize}[leftmargin=2.2em]
\item[$(F)$] $f:[0,\infty)\to\R$ is continuous, positive on $(0,\infty)$ and locally Lipschitz
continuous in $(0,\infty)$.
\end{itemize}
When a boundary condition is relevant we shall consider, as in \cite{DS1,DS2}, the Dirichlet problem
\begin{equation}\label{eq:dirichlet}
-\Lap{p}(u)=f(u)\ \text{ in }\Omega,\qquad u>0\ \text{ in }\Omega,\qquad u=0\ \text{ on }\partial\Omega,
\end{equation}
in a bounded smooth domain $\Omega$, with $u\in C^1(\overline\Omega)$.

The operator $\Lap{p}$ is singular ($1<p<2$) or degenerate ($p>2$) on the critical set
\begin{equation}\label{eq:Z}
Z_u:=\{x\in\Omega:\ Du(x)=0\},
\end{equation}
and solutions of \eqref{eq:main} are in general only of class $C^{1,\tau}$ \cite{DiB,Tol,Lie}. Two
of the basic tools of linear elliptic theory, the strong maximum principle for the linearized
operator and the strong comparison principle for the equation, are therefore delicate, and both
are classical only away from $Z_u$, where the operator is uniformly elliptic with continuous
coefficients and the linear theory applies, see \cite{Dam1,PucciSerrin}. Across $Z_u$ they may
genuinely fail, and a counterexample to the strong comparison principle (with
the weak one in force) is given in \cite{CT}. Partial results under structural restrictions on the
nonlinearity were obtained in \cite{CT,LP}, for sign-changing $f$ in \cite{RS}, and we refer to
\cite{Vaz} for the strong maximum principle in the quasilinear setting and to \cite{DamPac}
for the moving plane method in the singular case $1<p<2$.

\subsection{The state of the art and the restriction \texorpdfstring{$p>\frac{2N+2}{N+2}$}{p>(2N+2)/(N+2)}}
In the fundamental papers \cite{DS1,DS2}, Damascelli and Sciunzi introduced the point of view that
has dominated the subject ever since: one regards
\begin{equation}\label{eq:rho}
\rho:=|Du|^{p-2}
\end{equation}
as a \emph{weight}, and one studies the linearized operator
\begin{equation}\label{eq:linearized}
L_u(v,\varphi):=\int_\Omega\Big[|Du|^{p-2}(Dv,D\varphi)+(p-2)|Du|^{p-4}(Du,Dv)(Du,D\varphi)\Big]dx
-\int_\Omega f'(u)v\varphi\,dx
\end{equation}
in the weighted Sobolev space $H^{1,2}_{\rho}$. Since
\begin{equation}\label{eq:ellipticity}
\min\{1,p-1\}\,\rho\,|\xi|^2\ \le\ |Du|^{p-2}|\xi|^2+(p-2)|Du|^{p-4}(Du,\xi)^2\ \le\ \max\{1,p-1\}\,\rho\,|\xi|^2 ,
\end{equation}
$L_u$ is a linear operator whose ellipticity is comparable, from above and from below, with the
\emph{scalar} weight $\rho$. It is therefore a degenerate (or singular) elliptic operator of the kind
studied by Trudinger \cite{Tru1} and, in the Muckenhoupt setting, by Fabes, Kenig and Serapioni
\cite{FKS}; the difficulty is that the weight $\rho$ is not given a priori but is produced by the
solution itself. The crucial regularity input of \cite{DS1} is the summability of
negative powers of $|Du|$: for $u$ as in \eqref{eq:main} with $f$ satisfying $(F)$,
\begin{equation}\label{eq:DSsummability}
\int_{\Omega'}\frac{1}{|Du|^{(p-1)r}}\,\frac{1}{|x-y|^{\gamma}}\,dy\le C\qquad
\text{uniformly in }x ,
\end{equation}
for every $r<1$, every $\gamma<N-2$ (and $\gamma=0$ if $N=2$), and every $\Omega'\Subset\Omega$.
Using \eqref{eq:DSsummability} and the version, due to Trudinger \cite{Tru2}, of Moser's iteration
\cite{Moser}, the authors
of \cite{DS2} proved a weak Harnack inequality for $L_u$, a Harnack comparison inequality for two
solutions, a strong maximum principle for $L_u$, a description of $Z_u$ in convex symmetric domains
and, finally, the following strong comparison principle
(\cite[Theorem 1.4]{DS2}).

\medskip
\noindent\textbf{Theorem A (Damascelli--Sciunzi).}
\emph{Let $u,v\in C^1(\Omega)$, $\Omega\subset\R^N$ a bounded smooth domain, with
$\frac{2N+2}{N+2}<p<2$ or $p>2$. Assume that either $u$ or $v$ is a weak solution of
\eqref{eq:dirichlet} with $f$ satisfying $(F)$, and that}
\[
-\Lap{p}(u)+\Lambda u\le-\Lap{p}(v)+\Lambda v,\qquad u\le v \quad\text{in }\Omega ,
\]
\emph{for some $\Lambda\in\R$. Then $u\equiv v$ in $\Omega$, unless $u<v$ in $\Omega$.}
\medskip

The asymmetry between the two admissible ranges is striking. For $p>2$ the result is unrestricted;
for $1<p<2$ one is confined to
\[
\frac{2N+2}{N+2}=2-\frac{2}{N+2}<p<2 ,
\]
an interval whose length tends to $0$ as $N\to\infty$: already for $N=10$ one needs $p>1.8\overline{3}$,
and for $N=100$ one needs $p>1.98$. There is no structural reason to expect the singular
range to depend on the dimension in this way, and the restriction is universally regarded as
technical.

To the best of our knowledge, this restriction has never been relaxed. It reappears verbatim in
Sciunzi \cite[Theorems 1.3--1.4]{Sci2}, where the strong comparison principle is extended to
possibly vanishing source terms but where one of the two functions is still required to solve the
equation and the exponent is still constrained to $\frac{2N+2}{N+2}<p<\infty$; in Merchán--Montoro--Sciunzi \cite[Theorems 1.1--1.2]{MMS}, where
first-order terms are allowed; and it is inherited by all the applications of these Harnack-type
inequalities, see e.g. \cite{FMS,Gatti,Sci1} and Section \ref{sec:applications} below. The reason
is always the same, and can be read off
from \cite[Case (b) of Theorem 3.1]{MMS}: in the singular case one uses the \emph{unweighted}
Sobolev inequality on the left-hand side of the Caccioppoli inequality (which is legitimate because
$\rho\ge\lambda_0>0$ when $|Du|$ is bounded), but then the cut-off term
\[
\int_\Omega \rho\, w^2 |D\eta|^2\,dx
\]
on the right-hand side has to be treated by H\"older's inequality, and a gain of integrability is
obtained \emph{only if} $\rho\in L^t$ with
\begin{equation}\label{eq:oldcondition}
t>\frac{N}{2}.
\end{equation}
By \eqref{eq:DSsummability}, $\rho=|Du|^{-(2-p)}\in L^t_{\rm loc}$ precisely for
$t<\frac{p-1}{2-p}$, and $\frac{p-1}{2-p}>\frac N2$ if and only if $p>\frac{2N+2}{N+2}$. The
dimension enters only through \eqref{eq:oldcondition}.

\subsection{Main results}
The purpose of this paper is to show that condition \eqref{eq:oldcondition} is an artefact of the
way the Caccioppoli inequality is exploited, and that it can be replaced by a Morrey-type condition
on the measure $\rho\,dx$ which, for solutions of \eqref{eq:main}, is satisfied as soon as
$\rho\in L^{t}_{\rm loc}$ for some $t>1$, that is, as soon as the weight is locally integrable in
the mildest quantitative sense, which is the minimal requirement for the whole weighted framework to
make sense at all. We obtain the following
dimension-free statements. Throughout, $B_r(x)$ denotes the open ball of radius $r$ centred at $x$,
and $\Omega'\Subset\Omega$ means $\overline{\Omega'}$ is compact and contained in $\Omega$.

\begin{theorem}[Weak Harnack inequality for the linearized operator]\label{thm:main1}
Let $\Omega\subseteq\R^N$ be a domain, $N\ge2$, let $f$ satisfy $(F)$, let
\[
\tfrac32<p<2 ,
\]
and let $u\in C^1(\Omega)$ be a weak solution of \eqref{eq:main}. Set $\rho=|Du|^{p-2}$ and
$d\mu=\rho\,dx$. Let $\Omega'\Subset\Omega$ and $\overline{B_{6\delta}(x)}\subset\Omega'$. Then there
exist $\chi>1$ and, for every $0<s<\chi$, a constant $C>0$ such that for every nonnegative bounded
weak supersolution $v\in H^{1,2}_{\rho}(\Omega')$ of $L_u(v,\cdot)=0$,
\begin{equation}\label{eq:weakharnack1}
\Big(\int_{B_{2\delta}(x)}v^{s}\,d\mu\Big)^{1/s}\le C\,\inf_{B_{\delta}(x)}v .
\end{equation}
In particular the same inequality holds with $d\mu$ replaced by $dx$.
\end{theorem}

\begin{theorem}[Strong maximum principle for the linearized operator]\label{thm:main2}
Let $\Omega$, $f$, $p$ and $u$ be as in Theorem \ref{thm:main1}. Let $v\in H^{1,2}_{\rho,\rm loc}(\Omega)\cap C^0(\Omega)$
be a weak supersolution of $L_u(v,\cdot)=0$. Then for every \emph{connected} subdomain
$\Omega_0\subseteq\Omega$ with $v\ge0$ in $\Omega_0$ we have either $v\equiv0$ in $\Omega_0$ or
$v>0$ in $\Omega_0$.

In particular, for every $i\in\{1,\dots,N\}$ and every connected subdomain $\Omega_0$ with
$u_{x_i}\ge0$ in $\Omega_0$, either $u_{x_i}\equiv0$ or $u_{x_i}>0$ in $\Omega_0$.
\end{theorem}

\begin{theorem}[Strong comparison principle]\label{thm:main4}
Let $\Omega\subseteq\R^N$ be a domain, $N\ge2$, and let $\frac32<p<2$.
Let $u,v\in C^1(\Omega)$ and assume that either $u$ or $v$ is a weak solution of \eqref{eq:main}
with $f$ satisfying $(F)$. Assume that, for some $\Lambda\in\R$,
\begin{equation}\label{eq:scp-hyp}
-\Lap{p}(u)+\Lambda u\ \le\ -\Lap{p}(v)+\Lambda v,\qquad u\le v\quad\text{in }\Omega
\end{equation}
in the weak distributional sense. Then, in every connected subdomain $\Omega_0\subseteq\Omega$,
either $u\equiv v$ or $u<v$.

The same conclusion holds if \eqref{eq:scp-hyp} is replaced by
$-\Lap{p}(u)-f(u)\le-\Lap{p}(v)-f(v)$ and $u\le v$ in $\Omega$.
\end{theorem}

\begin{theorem}[Structure of the critical set]\label{thm:main3}
Let $\Omega\subset\R^N$ be a bounded smooth domain, $N\ge2$, let $\frac32<p<2$, let $f$ be positive
and locally Lipschitz continuous on $[0,\infty)$, and let $u\in C^1(\overline\Omega)$ be a weak
solution of \eqref{eq:dirichlet}. If $\Omega$ is convex in the directions $e_1,\dots,e_N$ and
symmetric with respect to the hyperplanes $\{x_i=0\}$, $i=1,\dots,N$, then
\[
Z_u=\{x\in\Omega:\ Du(x)=0\}=\{0\}, \qquad\text{and consequently}\qquad u\in C^2(\Omega\setminus\{0\}).
\]
\end{theorem}

Theorems \ref{thm:main1}--\ref{thm:main3} improve, respectively, Theorems 1.1, 1.2, 1.4 and 1.3 of
\cite{DS2}: the range $\frac{2N+2}{N+2}<p<2$ is replaced by $\frac32<p<2$ in every dimension. We have stated
all four results only in the singular range, since for $p>2$ the corresponding statements are those
of \cite[Theorems 1.1--1.4]{DS2} and already carry no restriction on the exponent; the degenerate
case is not touched here. Since
$\frac{2N+2}{N+2}=\frac32$ if and only if $N=2$, the new threshold is exactly the one that
\cite{DS2} obtained in the plane, now valid in all dimensions. Moreover Theorems \ref{thm:main1},
\ref{thm:main2} and \ref{thm:main4} are stated for \emph{local} solutions on an arbitrary domain,
with no boundary condition and no boundedness assumption on $\Omega$; this is a further (minor)
improvement over \cite{DS2}, where the Dirichlet problem in a bounded smooth domain is assumed
throughout.

\subsection{The mechanism, and an abstract criterion}
The proof of Theorems \ref{thm:main1}--\ref{thm:main3} rests on two facts. The first is a new
regularity estimate, which we believe to be of independent interest.

\begin{theorem}[Uniform Morrey estimate for negative powers of the gradient]\label{thm:morrey}
Let $\Omega\subseteq\R^N$ be a domain, $N\ge2$, $1<p<2$, let $f$ satisfy $(F)$ and let
$u\in C^1(\Omega)$ be a weak solution of \eqref{eq:main}. Let $\Omega''\Subset\Omega'\Subset\Omega$
and let $0\le\sigma<p-1$. Then there exist $\kappa>0$, $\bar r>0$ and $C>0$ such that
\begin{equation}\label{eq:morrey}
\int_{B_r(x_0)}\frac{dy}{|Du(y)|^{\sigma}}\ \le\ C\,r^{\,N-2+\kappa}
\qquad\text{for every }x_0\in\Omega''\text{ and every }0<r<\bar r .
\end{equation}
One may take $\bar r=\min\{1,\frac18\dist(\Omega'',\partial\Omega')\}$ and
$\kappa=\min\{1,\tau(p-\beta)\}$, where $\tau\in(0,1]$ is the H\"older exponent of $Du$ on
$\overline{\Omega'}$ and $\beta:=\max\{0,\sigma+2-p\}<1$; the constant $C$ depends only on
$N,p,\sigma$ and on
\[
\|Du\|_{L^{\infty}(\Omega')},\qquad [Du]_{C^{0,\tau}(\overline{\Omega'})},\qquad
\min_{\overline{\Omega'}}f(u),\qquad \sup_{\overline{\Omega'}}|f'(u)| .
\]
\end{theorem}

The two subdomains play different roles: the estimate is asserted for centres $x_0$ in the inner
subdomain $\Omega''$, while $\Omega'$ is the (slightly larger) set on which the four quantities above
are measured. The restriction $r<\bar r$ guarantees that
$\overline{B_{4r}(x_0)}\subset\Omega'$ for every $x_0\in\Omega''$, so that the whole argument takes
place inside $\Omega'$; in particular the ball $B_r(x_0)$ in \eqref{eq:morrey} is always contained in
$\Omega'$.

The point is the \emph{exponent} $N-2+\kappa>N-2$, and the \emph{uniformity in the centre $x_0$}.
Estimate \eqref{eq:DSsummability} gives, by taking $y=x_0$ and bounding $|x-x_0|^{-\gamma}\ge
r^{-\gamma}$ on $B_r(x_0)$, only
$\int_{B_r(x_0)}|Du|^{-\sigma}\le Cr^{\gamma}$ with $\gamma<N-2$; the threshold $N-2$ is exactly the
one that must be crossed. (For $x_0\in Z_u$ a bound of the type \eqref{eq:morrey} can be extracted
from the refined estimates of \cite{Sci2}, whose test functions carry the extra factor
$|x-x_0|^{-\mu}$. The novelty here is that \eqref{eq:morrey} holds \emph{uniformly over all
centres}, which is what a trace inequality requires, and that the proof is elementary and
self-contained.)

The second fact is that a Morrey bound of the form \eqref{eq:morrey} is exactly what is needed for a
\emph{trace inequality of Adams' type} \cite{Adams73,AdamsHedberg}, and that such a trace inequality
upgrades to a Sobolev inequality carrying the weight $\rho$ on both sides. This allows one to run
Moser's iteration in the measure $\mu=\rho\,dx$, so that the offending cut-off term
$\int\rho\,w^2|D\eta|^2\,dx$ becomes a perfectly harmless $\int w^2(\eta+|D\eta|)^2\,d\mu$. We
isolate the mechanism as follows.

\begin{definition}\label{def:Mkappa}
Let $\Omega'\subset\R^N$ be a bounded domain, $\lambda_0>0$, $\Lambda_0>0$, $\kappa>0$. We say that a
measurable weight $\rho:\Omega'\to(0,\infty]$ satisfies condition $(\mathcal M_\kappa)$ in $\Omega'$
if
\[
\rho\ge\lambda_0\ \text{ a.e. in }\Omega',\qquad\text{and}\qquad
\mu(B_r(x))\le \Lambda_0\, r^{\,N-2+\kappa}\ \ \ \forall x\in\R^N,\ \forall r>0 ,
\]
where $\mu$ is the measure $\rho\,\chi_{\Omega'}\,dx$.
\end{definition}

\begin{theorem}[Abstract weak Harnack comparison inequality]\label{thm:abstract}
Let $u,v\in C^1(\Omega)$ with $u\le v$ and assume \eqref{eq:scp-hyp}. Let $\Omega'\Subset\Omega$ and
suppose that the weight $\rho:=(|Du|+|Dv|)^{p-2}$ satisfies $(\mathcal M_\kappa)$ in $\Omega'$ for
some $\kappa>0$. Then there exists $\chi>1$ such that, whenever
$\overline{B_{6\delta}(x)}\subset\Omega'$ and $0<s<\chi$, there is a constant $C>0$ with
\[
\Big(\int_{B_{2\delta}(x)}(v-u)^s\,d\mu\Big)^{1/s}\le C\,\inf_{B_{\delta}(x)}(v-u).
\]
Consequently the strong comparison principle holds in $\Omega'$.
\end{theorem}

Theorem \ref{thm:main4} follows by combining Theorem \ref{thm:abstract} with Theorem
\ref{thm:morrey}: for $\frac32<p<2$ one has $\sigma:=2-p<p-1$, so \eqref{eq:morrey} applies to
$\rho=|Du|^{-(2-p)}$ and, a fortiori, to $(|Du|+|Dv|)^{p-2}\le|Du|^{p-2}$.

\subsection{Applications}
The strong comparison principle and the strong maximum principle for the linearized operator are
used as black boxes in a substantial part of the qualitative theory of $p$-Laplace equations, and
the restriction $\frac{2N+2}{N+2}<p<2$ propagates to every such result. In Section
\ref{sec:applications} we make this systematic and obtain, in particular: Gibbons' conjecture for
$-\Delta_pu=f(u)$ in $\R^N$ (Esposito--Farina--Montoro--Sciunzi \cite{EFMS1}) for
$\frac32<p<2$ (Theorem \ref{thm:gibbons}); the monotonicity in half-spaces for changing-sign
nonlinearities of \cite{EFMS2} for $\frac32<p<2$ (Theorem \ref{thm:halfspace}); the Harnack
comparison inequality with a first-order term of Merch\'an--Montoro--Sciunzi \cite{MMS} and the
strong comparison and maximum principles with a gradient term of \cite{Le}, again for
$\frac32<p<2$ (Theorem \ref{thm:firstorder}). Along the way we
also relax hypothesis $(F)$: only the constant sign of the composition $f(u)$ on the subdomain
under consideration is needed (Proposition \ref{prop:H}), which is what the applications actually
provide.

\subsection{An abstract criterion in action}
The abstract criterion is flexible, and gives more than $p>\frac32$ when extra information on the
solution is available. As an illustration:

\begin{corollary}[The radial case]\label{cor:ball}
Let $\Omega=B_R(0)\subset\R^N$, $N\ge2$, let $f$ satisfy $(F)$ and let $u\in C^1(\overline\Omega)$ be
a weak solution of \eqref{eq:dirichlet}. If $\frac43<p<2$, then the conclusions of Theorems \ref{thm:main1}, \ref{thm:main2} and
\ref{thm:main4} hold in $\Omega$.
\end{corollary}

\subsection{Optimality of the threshold \texorpdfstring{$\frac32$}{3/2}}
Finally we show that $\frac32$ is not an artefact of our proof, but the natural limit of the
weighted-space approach altogether.

\begin{proposition}\label{prop:sharp}
Let $1<p<2$ and let $p'=\frac{p}{p-1}$. For $c>0$ set
\begin{equation}\label{eq:example}
u(x)=c-\frac{|x_1|^{p'}}{p'},\qquad x\in\R^N .
\end{equation}
Then $u\in C^1(\R^N)$, $-\Lap{p}(u)=1$ in $\R^N$, and $u>0$ on any ball $B$ on which
$|x_1|^{p'}<p'c$; thus $u$ is a positive solution of \eqref{eq:main} on $B$ with $f\equiv1$
satisfying $(F)$. Moreover $|Du(x)|=|x_1|^{\frac1{p-1}}$, $Z_u=\{x_1=0\}$ and:
\begin{enumerate}[label=\rm(\roman*)]
\item $\int_B|Du|^{-\sigma}dx<\infty$ if and only if $\sigma<p-1$; hence the summability exponent in
\eqref{eq:DSsummability} is optimal;
\item $\rho=|Du|^{p-2}\in L^1(B)$ if and only if $p>\frac32$. In particular, for
$1<p\le\frac32$ the weight $\rho$ need not be locally integrable, and the weighted space
$H^{1,2}_{\rho}$ is not defined;
\item if $p>\frac32$, then $\int_{B_r(x_0)}\rho\,dx\simeq r^{\,N-\theta}$ for $x_0\in Z_u$, where
$\theta:=\frac{2-p}{p-1}\in(0,1)$; hence $\rho$ satisfies $(\mathcal M_\kappa)$ with
$\kappa=2-\theta>1$. If $p\le\frac32$ that integral is $+\infty$ for every $r>0$.
\end{enumerate}
\end{proposition}

The solution \eqref{eq:example} is a local solution of \eqref{eq:main}, which is the setting of
Theorems \ref{thm:main1}, \ref{thm:main2} and \ref{thm:main4}. One might hope that the phenomenon
disappears for solutions of the Dirichlet problem \eqref{eq:dirichlet} in a bounded smooth domain,
whose critical set could conceivably always be small. It does not: it suffices to look at an
annulus.

\begin{proposition}\label{prop:annulus}
Let $1<p<2$, let $0<a<b$ and let $\Omega:=\{x\in\R^N:\ a<|x|<b\}$, a bounded smooth domain. Let
$u\in C^1(\overline\Omega)$ be the unique weak solution of
\begin{equation}\label{eq:annulus}
-\Lap{p}(u)=1\ \text{ in }\Omega,\qquad u=0\ \text{ on }\partial\Omega .
\end{equation}
Then $u>0$ in $\Omega$, $u$ is radially symmetric, and there is a unique $r_0\in(a,b)$ such that
\[
Z_u=\{x:\ |x|=r_0\},\qquad
\lim_{r\to r_0}\frac{|Du(x)|}{\big||x|-r_0\big|^{\frac1{p-1}}}=1\quad (r=|x|) ,
\]
while $|Du|$ is bounded away from $0$ outside any neighbourhood of $Z_u$. Consequently the critical
set of $u$ is a smooth hypersurface, and, with $\theta:=\frac{2-p}{p-1}$,
\[
\rho=|Du|^{p-2}\in L^1(\Omega)\iff p>\tfrac32,\qquad
\int_{B_r(x_0)}\rho\,dx\simeq r^{\,N-\theta}\ \ (x_0\in Z_u)\ \text{ when }p>\tfrac32 ,
\]
whereas for $1<p\le\frac32$ one has $\int_{B_r(x_0)}\rho\,dx=+\infty$ for every $x_0\in Z_u$ and
every $r>0$.
\end{proposition}

Thus the obstruction is not an artefact of working locally: already for the model problem
\eqref{eq:annulus} on a smooth bounded domain, the weight fails to be integrable as soon as
$p\le\frac32$. Incidentally, Proposition \ref{prop:annulus} also shows that the convexity assumption
cannot be removed from Theorem \ref{thm:main3}: on an annulus the critical set of a solution of
\eqref{eq:dirichlet} is a hypersurface, not a point.

Item (ii) of Proposition \ref{prop:sharp} shows that no argument formulated in $H^{1,2}_{\rho}$ can
reach below $p=\frac32$; item
(iii) shows that, in the same example, our hypothesis $(\mathcal M_\kappa)$ holds in exactly the
range $p>\frac32$, so that Theorem \ref{thm:abstract} is being applied under an optimal assumption.
The better threshold $\frac43$ of Corollary \ref{cor:ball} reflects the fact that there the critical
set is a single point rather than a hypersurface; see Remark \ref{rem:minkowski}. Whether the strong
comparison principle survives for $1<p\le\frac32$ is, to our knowledge, completely open; we discuss
this in Section \ref{sec:remarks}. Let us stress that Propositions \ref{prop:sharp} and
\ref{prop:annulus} are not counterexamples to the strong comparison principle: they only certify
that a genuinely different technique, one that does not linearize in $H^{1,2}_{\rho}$, would
be required.

\subsection{Organization} Section \ref{sec:prelim} collects notation and known results. Section
\ref{sec:morrey} contains the proof of Theorem \ref{thm:morrey}. Section \ref{sec:sobolev} develops
the weighted Sobolev inequality with gain and the Moser iteration in the measure $\mu$, proving
Theorem \ref{thm:abstract}. Section \ref{sec:proofs} contains the proofs of the main theorems and of
Corollary \ref{cor:ball}. Section \ref{sec:applications} is devoted to the applications, and
Section \ref{sec:remarks} to Propositions \ref{prop:sharp} and \ref{prop:annulus} and to some open
questions.

%%%%%%%%%%%%%%%%%%%%%%%%%%%%%%%%%%%%%%%%%%%%%%%%%%%%%%%%%%%%%%%%%%%%%%%%%%%%%%%
\section{Notation and preliminaries}\label{sec:prelim}
%%%%%%%%%%%%%%%%%%%%%%%%%%%%%%%%%%%%%%%%%%%%%%%%%%%%%%%%%%%%%%%%%%%%%%%%%%%%%%%

Throughout, $C$ denotes a positive constant which may change from line to line and whose dependence
is indicated when relevant. We write $u_{x_i}=\partial u/\partial x_i$, $Du=(u_{x_1},\dots,u_{x_N})$,
$D^2u$ for the Hessian and $\|D^2u\|^2=\sum_{i}|Du_{x_i}|^2$.

For a weight $0<\rho\in L^1(\Omega')$ and $1\le s<\infty$, $H^{1,s}_{\rho}(\Omega')$ is the space of
functions with distributional gradient for which
$\|v\|_{H^{1,s}_\rho}=\|v\|_{L^s(\Omega')}+\|Dv\|_{L^s(\Omega',\rho)}$ is finite, where
$\|Dv\|^s_{L^s(\Omega',\rho)}=\int_{\Omega'}\rho|Dv|^s$; $H^{1,s}_{0,\rho}(\Omega')$ is the closure of
$C_c^\infty(\Omega')$ in that norm. See \cite{DS1,HKM,Tru2}.

We shall systematically use the following elementary facts. If $1<p<2$ and
$M:=\|Du\|_{L^\infty(\Omega')}<\infty$, then, since $p-2<0$,
\begin{equation}\label{eq:lowerweight}
\rho=|Du|^{p-2}\ \ge\ M^{p-2}=:\lambda_0>0\qquad\text{a.e. in }\Omega' .
\end{equation}
Consequently $\mu:=\rho\,dx$ and the Lebesgue measure are mutually absolutely continuous on
$\Omega'$, $H^{1,2}_{0,\rho}(\Omega')\subseteq H^{1,2}_0(\Omega')$, and essential infima and suprema
with respect to $\mu$ and to $dx$ coincide.

We recall the standard vector inequalities (see e.g. \cite[Lemma 2.1]{Dam1}): for every $p>1$ there
are $\hat C,\check C>0$ such that for all $\eta,\eta'\in\R^N$ with $|\eta|+|\eta'|>0$,
\begin{equation}\label{eq:vector}
\big[|\eta|^{p-2}\eta-|\eta'|^{p-2}\eta'\big]\cdot(\eta-\eta')\ \ge\ \hat C(|\eta|+|\eta'|)^{p-2}|\eta-\eta'|^2 ,
\end{equation}
\begin{equation}\label{eq:vector2}
\big||\eta|^{p-2}\eta-|\eta'|^{p-2}\eta'\big|\ \le\ \check C(|\eta|+|\eta'|)^{p-2}|\eta-\eta'| .
\end{equation}

Finally, in the proof of Theorem \ref{thm:main3} we shall use the notation of the moving plane
method, which we recall from \cite[\S2]{DS2}. For a direction $\nu\in S^{N-1}$ and $\lambda\in\R$ set
\[
T^\nu_\lambda:=\{x\in\R^N:\ x\cdot\nu=\lambda\},\qquad
\Omega^\nu_\lambda:=\{x\in\Omega:\ x\cdot\nu<\lambda\},\qquad
x^\nu_\lambda:=x+2(\lambda-x\cdot\nu)\nu ,
\]
let $a(\nu):=\inf_{x\in\Omega}x\cdot\nu$ and
$(\Omega^\nu_\lambda)':=\{x^\nu_\lambda:\ x\in\Omega^\nu_\lambda\}$. The \emph{optimal cap} is
$\Omega^\nu_{\lambda_2(\nu)}$, where
\[
\lambda_2(\nu):=\sup\big\{\lambda>a(\nu):\ (\Omega^\nu_\mu)'\subseteq\Omega\ \text{ for all }\mu\in(a(\nu),\lambda]\big\} .
\]
If $\Omega$ is convex in the direction $\nu$ and symmetric with respect to $T^\nu_0$, then
$\lambda_2(\nu)=\lambda_2(-\nu)=0$.

We shall use the following known results.

\begin{theorem}[\cite{DiB,Tol,Lie,Teix}]\label{thm:C1alpha}
Let $u\in C^1(\Omega)$ be a weak solution of \eqref{eq:main} with $f$ satisfying $(F)$. Then
$u\in C^{1,\tau}_{\rm loc}(\Omega)$ for some $\tau\in(0,1]$. In particular, for every
$\Omega'\Subset\Omega$ there is $L>0$ with
\begin{equation}\label{eq:holder}
|Du(x)-Du(y)|\le L|x-y|^{\tau}\qquad \forall x,y\in\overline{\Omega'} .
\end{equation}
\end{theorem}

\begin{theorem}[{\cite[Theorem 1.1 and (1.2)]{DS1}}]\label{thm:DSlin}
Let $u\in C^1(\Omega)$ be a weak solution of \eqref{eq:main} with $f$ satisfying $(F)$. Then
$u_{x_i}\in H^{1,2}_{\rho,\rm loc}(\Omega)$, $|Du|^{p-2}Du\in W^{1,2}_{\rm loc}(\Omega;\R^N)$, and
\begin{equation}\label{eq:linear-eq}
L_u(u_{x_i},\varphi)=0\qquad\text{for every }\varphi\in W^{1,2}(\Omega)\ \text{with compact support in }\Omega,
\end{equation}
for $i=1,\dots,N$, where $L_u$ is as in \eqref{eq:linearized}. Moreover, for every
$\Omega'\Subset\Omega$, every $\beta<1$ and every $\gamma<N-2$ ($\gamma=0$ if $N=2$),
\[
\int_{\Omega'}\frac{|Du|^{p-2-\beta}}{|x-y|^\gamma}\|D^2u\|^2\,dx\le C
\qquad\text{uniformly for }y\in\Omega' .
\]
Finally $|Z_u|=0$ \emph{(\cite{Lou}; it also follows from Theorem \ref{thm:morrey} above)}.
\end{theorem}

\begin{remark}\label{rem:global}
Estimate \eqref{eq:DSsummability} is proved in \cite[Theorem 1.1]{DS1} under the global hypotheses
that $\Omega$ be bounded and smooth and that $u\in C^1(\overline\Omega)$ solve the Dirichlet problem
\eqref{eq:dirichlet}. \emph{We shall never use it}: Theorem \ref{thm:morrey} below reproves it, for
local solutions and with the additional Morrey information, from Theorem \ref{thm:DSlin} alone,
whose statements are purely local. This is why Theorems \ref{thm:main1}, \ref{thm:main2} and
\ref{thm:main4} can be stated for local solutions on an arbitrary domain.
\end{remark}

\begin{remark}\label{rem:W22}
If $1<p<2$ then, by Theorem \ref{thm:DSlin} with $\beta=\gamma=0$ and by \eqref{eq:lowerweight},
\[
\int_{\Omega'}\|D^2u\|^2\,dx\le\lambda_0^{-1}\int_{\Omega'}|Du|^{p-2}\|D^2u\|^2\,dx<\infty ,
\]
so that $u\in W^{2,2}_{\rm loc}(\Omega)$. This will be used freely to justify the choice of test
functions involving $|Du|$.
\end{remark}

Finally we recall the trace inequality that we shall use. For $0<\alpha<N$ let
$I_\alpha g(x)=\int_{\R^N}|x-y|^{\alpha-N}g(y)\,dy$ denote the Riesz potential.

\begin{theorem}[Adams \cite{Adams73}; see also {\cite[Theorem 7.2.1]{AdamsHedberg}} and {\cite[\S1.4]{Mazya}}]\label{thm:adams}
Let $N\ge3$, $2<q<\infty$ and let $\nu$ be a nonnegative Borel measure on $\R^N$. Then
\[
\|I_1g\|_{L^q(\nu)}\le A\|g\|_{L^2(\R^N)}\qquad\forall\, 0\le g\in L^2(\R^N)
\]
holds if (and only if) there is $\Lambda_0>0$ with
$\nu(B_r(x))\le\Lambda_0\,r^{\frac{q(N-2)}{2}}$ for all $x\in\R^N$, $r>0$; moreover $A$ depends only
on $N,q,\Lambda_0$.
\end{theorem}

%%%%%%%%%%%%%%%%%%%%%%%%%%%%%%%%%%%%%%%%%%%%%%%%%%%%%%%%%%%%%%%%%%%%%%%%%%%%%%%
\section{A uniform Morrey estimate: proof of Theorem \ref{thm:morrey}}\label{sec:morrey}
%%%%%%%%%%%%%%%%%%%%%%%%%%%%%%%%%%%%%%%%%%%%%%%%%%%%%%%%%%%%%%%%%%%%%%%%%%%%%%%

We first record two localized Caccioppoli-type inequalities, in a form which keeps track of the
size of $|Du|$ on the ball. Throughout this section $1<p<2$, $u\in C^1(\Omega)$ is a fixed weak
solution of \eqref{eq:main}, $f$ satisfies $(F)$, and $\Omega''\Subset\Omega'\Subset\Omega$ are
fixed. (For $p\ge2$ the weight $\rho=|Du|^{p-2}$ is bounded and none of the estimates of this
section is needed; the arguments below extend to $2\le p<3$ verbatim, since $u\in W^{2,2}_{\rm loc}$
also in that range by \cite{DS1,Sci1}, and to all $p>1$ after replacing $\|D^2u\|$ by
$|D(|Du|^{p-2}Du)|$ in the obvious way.) We set
\[
M:=\|Du\|_{L^\infty(\Omega')},\qquad
c_0:=\min_{\overline{\Omega'}}f(u)>0,\qquad
c_1:=\sup_{\overline{\Omega'}}|f'(u)|<\infty .
\]
That $c_0>0$ follows from $u>0$ in $\Omega$, the compactness of $\overline{\Omega'}$ and $(F)$;
$c_1<\infty$ follows from the local Lipschitz continuity of $f$ on $(0,\infty)$. All the integrands
below involving negative powers of $|Du|$ are defined a.e. in $\Omega'$ and are $+\infty$ on the
critical set $Z_u$, which is Lebesgue-negligible by Theorem \ref{thm:DSlin}; integrands of the form
$|Du|^{p-2-\beta}\|D^2u\|^2$ are understood to vanish on $Z_u$, which is legitimate since
$D^2u=0$ a.e. on $Z_u$ (the distributional gradient of a Sobolev function vanishes a.e. on each of
its level sets, applied to $u_{x_i}$ on $\{u_{x_i}=0\}$).

\begin{lemma}[Weighted second-order estimate]\label{lem:cacc2}
Let $0\le\beta<1$ and let $\psi\in C_c^\infty(\Omega')$, $0\le\psi\le1$. Then
\begin{equation}\label{eq:cacc2}
\int_{\Omega'}\psi^2|Du|^{p-2-\beta}\|D^2u\|^2\,dx\ \le\
C_1\left[\int_{\Omega'}|D\psi|^2|Du|^{p-\beta}\,dx+\int_{\Omega'}\psi^2|Du|^{2-\beta}\,dx\right],
\end{equation}
with $C_1=C_1(p,\beta,c_1)>0$.
\end{lemma}

\begin{proof}
For $\eps>0$ let
\[
G_\eps(t)=(2t-2\eps)\chi_{[\eps,2\eps]}(t)+t\,\chi_{(2\eps,\infty)}(t)\quad (t> 0),\qquad
G_\eps(t)=-G_\eps(-t)\quad (t<0),\qquad G_\eps(0)=0,
\]
and set $T_\eps(t)=G_\eps(t)|t|^{-\beta}$ (with $T_\eps(0)=0$). Then $T_\eps$ is Lipschitz,
$T_\eps(t)t\ge0$, $|T_\eps(t)|\le|t|^{1-\beta}$ and, for a.e. $t$,
\begin{equation}\label{eq:Teps}
T'_\eps(t)=\frac{1}{|t|^{\beta}}\Big(G'_\eps(t)-\beta\frac{G_\eps(t)}{|t|}\Big)
\ \ge\ \frac{1-\beta}{|t|^{\beta}}\,\chi_{\{|t|>\eps\}}\ \ge 0 ,
\end{equation}
because $G'_\eps=1=G_\eps(t)/|t|$ on $\{|t|>2\eps\}$ and $G'_\eps=2$,
$G_\eps(t)/|t|\le2$ on $\{\eps<|t|<2\eps\}$.

Fix $i$ and use $\varphi=\psi^2T_\eps(u_{x_i})$ in \eqref{eq:linear-eq}. This is admissible: by
Theorem \ref{thm:DSlin} and Remark \ref{rem:W22}, $u_{x_i}\in W^{1,2}_{\rm loc}$, hence
$\varphi\in W^{1,2}$ with compact support. Since
$D\varphi=\psi^2T'_\eps(u_{x_i})Du_{x_i}+2\psi T_\eps(u_{x_i})D\psi$, using \eqref{eq:ellipticity}
and \eqref{eq:Teps} we obtain
\begin{align*}
\min\{1,p-1\}(1-\beta)\int_{\{|u_{x_i}|>\eps\}}\psi^2\frac{|Du|^{p-2}|Du_{x_i}|^2}{|u_{x_i}|^{\beta}}
&\le 2\big(1+|p-2|\big)\int_{\Omega'}\psi|D\psi|\,|Du|^{p-2}|Du_{x_i}|\,|u_{x_i}|^{1-\beta}\\
&\qquad+c_1\int_{\Omega'}\psi^2|u_{x_i}|^{2-\beta} .
\end{align*}
Since $T_\eps$ vanishes identically on $\{|t|\le\eps\}$, the integral on the right-hand side is in
fact carried by $\{|u_{x_i}|>\eps\}$, where we may write
\[
\psi|D\psi||Du|^{p-2}|Du_{x_i}||u_{x_i}|^{1-\beta}
=\Big(\psi\frac{|Du|^{\frac{p-2}{2}}|Du_{x_i}|}{|u_{x_i}|^{\beta/2}}\Big)
\Big(|D\psi|\,|Du|^{\frac{p-2}{2}}|u_{x_i}|^{1-\frac\beta2}\Big),
\]
and apply Young's inequality, the first factor being absorbed into the left-hand side.
Since $|u_{x_i}|\le|Du|$ we get $|Du|^{p-2}|u_{x_i}|^{2-\beta}\le|Du|^{p-\beta}$ and
$|u_{x_i}|^{2-\beta}\le|Du|^{2-\beta}$, whence, after absorption,
\[
\int_{\{|u_{x_i}|>\eps\}}\psi^2\frac{|Du|^{p-2}|Du_{x_i}|^2}{|u_{x_i}|^{\beta}}\,dx\le
C_1\left[\int_{\Omega'}|D\psi|^2|Du|^{p-\beta}+\int_{\Omega'}\psi^2|Du|^{2-\beta}\right]
\]
with $C_1$ independent of $\eps$. Letting $\eps\to0$ and using Fatou's lemma, and then
$|u_{x_i}|^{-\beta}\ge|Du|^{-\beta}$ together with $Du_{x_i}=0$ a.e. on $\{u_{x_i}=0\}$, we obtain
\[
\int_{\Omega'}\psi^2|Du|^{p-2-\beta}|Du_{x_i}|^2\,dx\le
C_1\left[\int_{\Omega'}|D\psi|^2|Du|^{p-\beta}+\int_{\Omega'}\psi^2|Du|^{2-\beta}\right].
\]
Summing over $i=1,\dots,N$ gives \eqref{eq:cacc2}.
\end{proof}

\begin{lemma}[Localized summability inequality]\label{lem:cacc1}
Let $0\le\sigma<p-1$, put $\beta:=\max\{0,\sigma+2-p\}\in[0,1)$, and let $\psi\in C_c^\infty(\Omega')$
with $0\le\psi\le1$. Then
\begin{equation}\label{eq:cacc1}
\int_{\Omega'}\frac{\psi^2}{|Du|^{\sigma}}\,dx\ \le\ C_2\left[\int_{\Omega'}\psi|D\psi|\,|Du|^{p-1-\sigma}dx
+\int_{\Omega'}\psi^2|Du|^{p-2-\beta}\|D^2u\|^2dx\right],
\end{equation}
with $C_2=C_2(N,p,\sigma,c_0,M)>0$.
\end{lemma}

\begin{proof}
For $\eps>0$ set $g_\eps=(|Du|^2+\eps^2)^{-\sigma/2}$ and $\varphi=\psi^2g_\eps$. Since
$\xi\mapsto(|\xi|^2+\eps^2)^{-\sigma/2}$ is Lipschitz on $\R^N$ and $u\in W^{2,2}_{\rm loc}(\Omega)$
by Remark \ref{rem:W22}, $\varphi$ belongs to $W^{1,2}(\Omega)\cap L^\infty(\Omega)$ and has compact
support; as $|Du|^{p-2}Du\in L^\infty(\Omega')$, it is an admissible test function in
\eqref{eq:main}. Testing \eqref{eq:main} with $\varphi$ and using
\[
Dg_\eps=-\sigma(|Du|^2+\eps^2)^{-\frac\sigma2-1}D^2u\,Du,
\]
we obtain
\[
\int f(u)\psi^2g_\eps
=2\int\psi g_\eps|Du|^{p-2}(Du,D\psi)
-\sigma\int\psi^2(|Du|^2+\eps^2)^{-\frac\sigma2-1}|Du|^{p-2}(Du,D^2u\,Du),
\]
whence, using $g_\eps\le|Du|^{-\sigma}$ and $f(u)\ge c_0$,
\begin{equation}\label{eq:step1}
c_0\int\psi^2g_\eps\ \le\ 2\int\psi|D\psi|\,|Du|^{p-1-\sigma}
+\sigma\int\psi^2(|Du|^2+\eps^2)^{-\frac\sigma2-1}|Du|^{p}\,\|D^2u\| .
\end{equation}
(The first term is finite because $p-1-\sigma>0$ and $|Du|\le M$.) We estimate the last term by
Cauchy--Schwarz:
\[
\sigma\int\psi^2(|Du|^2+\eps^2)^{-\frac\sigma2-1}|Du|^{p}\|D^2u\|
\le\sigma\,\mathcal H^{1/2}\left(\int\psi^2|Du|^{p+2+\beta}(|Du|^2+\eps^2)^{-\sigma-2}\right)^{1/2},
\]
where $\mathcal H:=\int\psi^2|Du|^{p-2-\beta}\|D^2u\|^2$. Since $|Du|\le(|Du|^2+\eps^2)^{1/2}$,
\[
|Du|^{p+2+\beta}(|Du|^2+\eps^2)^{-\sigma-2}\le(|Du|^2+\eps^2)^{\frac{p+\beta-2\sigma-2}{2}} .
\]
The choice $\beta\ge\sigma+2-p$ gives $\frac{p+\beta-2\sigma-2}{2}\ge-\frac\sigma2$, so that, as
$(|Du|^2+\eps^2)^{1/2}\le(M^2+1)^{1/2}$ for $\eps\le1$,
\[
(|Du|^2+\eps^2)^{\frac{p+\beta-2\sigma-2}{2}}\le C(M,p,\sigma)\,(|Du|^2+\eps^2)^{-\frac\sigma2}=C\,g_\eps .
\]
Therefore
\[
c_0\int\psi^2g_\eps\le 2\int\psi|D\psi||Du|^{p-1-\sigma}+C\sigma\,\mathcal H^{1/2}\Big(\int\psi^2g_\eps\Big)^{1/2},
\]
and Young's inequality allows us to absorb the last term, giving
\[
\int\psi^2g_\eps\le C_2\Big[\int\psi|D\psi||Du|^{p-1-\sigma}+\mathcal H\Big]
\]
with $C_2$ independent of $\eps\in(0,1]$. Letting $\eps\downarrow0$, monotone convergence yields
\eqref{eq:cacc1}. Finally $\beta=\max\{0,\sigma+2-p\}<1$ because $\sigma<p-1$.
\end{proof}

\begin{proof}[Proof of Theorem \ref{thm:morrey}]
Let $\tau\in(0,1]$ and $L>0$ be as in \eqref{eq:holder} on $\overline{\Omega'}$; enlarging $L$ if
necessary we may and do assume $L\ge1$. Let $\beta=\max\{0,\sigma+2-p\}<1$ and set
\[
\bar r:=\min\Big\{1,\ \tfrac18\dist(\Omega'',\partial\Omega')\Big\} .
\]
Note that $\bar r\le1$, so that $r^{a}\le r^{b}$ whenever $a\ge b$ and $0<r<\bar r$; this will be used
repeatedly to compare the exponents below. Fix $x_0\in\Omega''$ and $0<r<\bar r$, so that $\overline{B_{4r}(x_0)}\subset\Omega'$. Denote
\[
J(r):=\int_{B_r(x_0)}\frac{dy}{|Du(y)|^{\sigma}},\qquad
S_r:=\max_{\overline{B_{4r}(x_0)}}|Du| .
\]

\emph{Step 1: a first bound.} Choose cut-off functions $\psi,\tilde\psi\in C_c^\infty(\R^N)$ with
\[
\psi\equiv1\text{ on }B_r(x_0),\quad \supp\psi\subset B_{2r}(x_0),\quad|D\psi|\le\tfrac2r,
\]
\[
\tilde\psi\equiv1\text{ on }B_{2r}(x_0),\quad \supp\tilde\psi\subset B_{4r}(x_0),\quad|D\tilde\psi|\le\tfrac2r .
\]
By Lemma \ref{lem:cacc1} and Lemma \ref{lem:cacc2} (the latter applied with $\tilde\psi$, which
dominates $\psi$),
\[
J(r)\le C_2\left[\frac2r\int_{B_{2r}(x_0)}|Du|^{p-1-\sigma}
+C_1\Big(\frac{4}{r^2}\int_{B_{4r}(x_0)}|Du|^{p-\beta}+\int_{B_{4r}(x_0)}|Du|^{2-\beta}\Big)\right].
\]
Since $p-1-\sigma>0$, $p-\beta>0$ and $2-\beta>0$, all the integrands are bounded by the
corresponding powers of $S_r$, and therefore
\begin{equation}\label{eq:firstbound}
J(r)\ \le\ C_3\Big[r^{N-1}S_r^{\,p-1-\sigma}+r^{N-2}S_r^{\,p-\beta}+r^{N}S_r^{\,2-\beta}\Big],
\end{equation}
with $C_3=C_3(N,p,\sigma,\beta,c_0,c_1,M)$.

\emph{Step 2: dichotomy.} We distinguish two cases.

\emph{Case 1: $|Du(x_0)|\ge 2L(4r)^{\tau}$.} By \eqref{eq:holder}, for every $y\in B_{4r}(x_0)$,
\[
|Du(y)|\ \ge\ |Du(x_0)|-L(4r)^{\tau}\ \ge\ \tfrac12|Du(x_0)|\ \ge\ L(4r)^{\tau}.
\]
Hence
\[
J(r)\le\big(L(4r)^{\tau}\big)^{-\sigma}|B_r|=C\,r^{\,N-\tau\sigma}.
\]
Since $\sigma<p-1<1$ and $\tau\le1$, we have $\tau\sigma<1$, so
\[
J(r)\le C\,r^{\,N-2+(2-\tau\sigma)},\qquad 2-\tau\sigma>1 .
\]

\emph{Case 2: $|Du(x_0)|<2L(4r)^{\tau}$.} Then, again by \eqref{eq:holder}, for $y\in B_{4r}(x_0)$,
\[
|Du(y)|\le|Du(x_0)|+L(4r)^{\tau}\le 3L(4r)^{\tau},\qquad\text{so } S_r\le C_4\,r^{\tau} .
\]
Plugging this into \eqref{eq:firstbound},
\[
J(r)\le C\Big[r^{\,N-1+\tau(p-1-\sigma)}+r^{\,N-2+\tau(p-\beta)}+r^{\,N+\tau(2-\beta)}\Big]
\le C\,r^{\,N-2+\min\{1,\,\tau(p-\beta)\}},
\]
where we used $r<\bar r$ and that the three exponents are not less than $N-2+1$, $N-2+\tau(p-\beta)$ and
$N-2+2$ respectively.

\emph{Step 3: conclusion.} In both cases
\[
J(r)\le C\,r^{\,N-2+\kappa},\qquad \kappa:=\min\{1,\ 2-\tau\sigma,\ \tau(p-\beta)\}
=\min\{1,\ \tau(p-\beta)\}>0 ,
\]
(the middle quantity being $>1$), with $C$ and $\kappa$ independent of $x_0\in\Omega''$ and of
$0<r<\bar r$. This is \eqref{eq:morrey}.
\end{proof}

\begin{corollary}\label{cor:summ}
Under the hypotheses of Theorem \ref{thm:morrey}, for every $\Omega''\Subset\Omega$ and every
$\sigma<p-1$ one has $\int_{\Omega''}|Du|^{-\sigma}dx<\infty$; in particular $|Z_u|=0$. Thus the
summability estimate of \cite[Theorem 1.1]{DS1} holds for local solutions, and with the additional
Morrey information \eqref{eq:morrey}.
\end{corollary}

\begin{proof}
Cover $\overline{\Omega''}$ by finitely many balls $B_{r}(x_0)$ with $x_0\in\Omega''$ and
$r<\bar r$, and sum \eqref{eq:morrey}. The set $Z_u$ is contained in $\{|Du|^{-\sigma}=+\infty\}$,
which is negligible for an integrable function.
\end{proof}

\begin{remark}\label{rem:comparison}
Theorem \ref{thm:morrey} should be compared with \eqref{eq:DSsummability}. Choosing $y=x_0$ in
\eqref{eq:DSsummability} and using $|x-x_0|\le r$ on $B_r(x_0)$ gives only
$\int_{B_r(x_0)}|Du|^{-\sigma}\le Cr^{\gamma}$ with $\gamma<N-2$, i.e. a Morrey exponent strictly
\emph{below} $N-2$; on the other hand H\"older's inequality combined with
$|Du|^{-\sigma}\in L^{t}$, $t<\frac{p-1}{\sigma}$, only gives the exponent $N(1-\frac1t)$, which
exceeds $N-2$ exactly when $t>\frac N2$, i.e. under the old condition \eqref{eq:oldcondition}. The
gain in Theorem \ref{thm:morrey} comes from the interplay between the equation (which forces
$|Du|$ to be small on \emph{small} balls centred near $Z_u$, at a rate controlled by the $C^{1,\tau}$
norm) and the second-order estimate \eqref{eq:cacc2}: the term $r^{-2}\int_{B_{4r}}|Du|^{p-\beta}$,
which is the only obstruction to an exponent larger than $N-2$, is precisely the term that becomes
small near the critical set.
\end{remark}

\begin{remark}
The proof shows that $\kappa$ can be taken to be $\min\{1,\tau(p-\beta)\}$, where $\tau$ is any
H\"older exponent of $Du$. Sharper values of $\tau$ (see \cite{Teix}, and \cite{Sci2} for the
consequences on the summability estimates) give sharper values of $\kappa$, hence a better exponent
$q$ in Theorem \ref{thm:sobolev} below, but they do not affect the range of $p$ for which our
results hold: only $\kappa>0$ matters.
\end{remark}

%%%%%%%%%%%%%%%%%%%%%%%%%%%%%%%%%%%%%%%%%%%%%%%%%%%%%%%%%%%%%%%%%%%%%%%%%%%%%%%
\section{Sobolev inequality with a gain, and Moser iteration in the measure \texorpdfstring{$\mu$}{mu}}\label{sec:sobolev}
%%%%%%%%%%%%%%%%%%%%%%%%%%%%%%%%%%%%%%%%%%%%%%%%%%%%%%%%%%%%%%%%%%%%%%%%%%%%%%%

Throughout this section $\Omega'\subset\R^N$ is a bounded domain and $\rho$ is a weight satisfying
condition $(\mathcal M_\kappa)$ of Definition \ref{def:Mkappa}, with constants
$\lambda_0,\Lambda_0,\kappa$. We write $d\mu=\rho\,dx$ on $\Omega'$, so that
\[
\|w\|_{L^{r}(\Omega',\mu)}^{\,r}=\int_{\Omega'}|w|^{r}\rho\,dx\qquad(1\le r<\infty);
\]
in particular $\|Dw\|_{L^2(\Omega',\mu)}^2=\int_{\Omega'}\rho\,|Dw|^2dx$ is the norm of the space
$H^{1,2}_{\rho}(\Omega')$, and the two symbols $\rho$ and $\mu$ refer to one and the same weight
throughout. Note that
$(\mathcal M_\kappa)$ forces $\mu(\Omega')<\infty$, i.e. $\rho\in L^1(\Omega')$.

\subsection{The Sobolev inequality}

\begin{theorem}\label{thm:sobolev}
Assume $\rho$ satisfies $(\mathcal M_\kappa)$ in $\Omega'$ and set
\begin{equation}\label{eq:qdef}
q:=\begin{cases}\ \dfrac{2(N-2+\kappa)}{N-2}=2+\dfrac{2\kappa}{N-2}, & N\ge3,\\[2mm]
\ 2+\kappa, & N=2 .\end{cases}
\end{equation}
Then $q>2$ and there exists $C_S=C_S(N,\kappa,\lambda_0,\Lambda_0,\Omega')>0$ such that
\begin{equation}\label{eq:sobolev}
\|w\|_{L^{q}(\Omega',\mu)}\ \le\ C_S\,\|Dw\|_{L^2(\Omega',\mu)}
\qquad\forall\,w\in H^{1,2}_{0,\rho}(\Omega') .
\end{equation}
Moreover, for every ball $B\subseteq\Omega'$ and every $w\in H^{1,2}_{\rho}(B)$,
\begin{equation}\label{eq:sobolev-mean}
\|w-\overline w_B\|_{L^{q}(B,\mu)}\ \le\ C_S'\,\|Dw\|_{L^2(B,\mu)},\qquad
\overline w_B:=\frac1{|B|}\int_Bw\,dx ,
\end{equation}
with $C_S'$ depending in addition on $\operatorname{diam}(\Omega')$ but not on $B$.
\end{theorem}

\begin{proof}
Assume first $N\ge3$. With $q$ as in \eqref{eq:qdef} we have $\frac{q(N-2)}{2}=N-2+\kappa$, so
$(\mathcal M_\kappa)$ is exactly the hypothesis of Theorem \ref{thm:adams} for the measure $\mu$.
For $w\in C_c^\infty(\Omega')$ the classical potential estimate gives
\[
|w(x)|\le\frac{1}{N\omega_N}\,I_1(|Dw|)(x)\qquad\forall x\in\R^N,
\]
where $\omega_N=|B_1|$ (this is elementary: extend $w$ by zero and integrate $Dw$ along rays). Hence, by Theorem \ref{thm:adams} applied with
$g=|Dw|\chi_{\Omega'}$,
\[
\|w\|_{L^q(\Omega',\mu)}\le \frac{A}{N\omega_N}\,\|Dw\|_{L^2(\Omega')}
\le \frac{A}{N\omega_N}\,\lambda_0^{-1/2}\,\|Dw\|_{L^2(\Omega',\mu)} ,
\]
where in the last step we used $\rho\ge\lambda_0$. Inequality \eqref{eq:sobolev} then follows for
all $w\in H^{1,2}_{0,\rho}(\Omega')$ by density, using again $\rho\ge\lambda_0$ (which guarantees
$H^{1,2}_{0,\rho}(\Omega')\subseteq H^{1,2}_0(\Omega')$ and that a Cauchy sequence in the weighted
norm is Cauchy in the unweighted one). For \eqref{eq:sobolev-mean} one starts instead from
\[
|w(x)-\overline w_B|\le\frac{(\operatorname{diam}B)^N}{N|B|}\int_B\frac{|Dw(y)|}{|x-y|^{N-1}}\,dy
\qquad\text{for a.e. }x\in B,
\]
see \cite[Lemma 7.16]{GT}, and argues in the same way.

If $N=2$ the exponent $p=2$ is not admissible in Theorem \ref{thm:adams}, and we argue as follows.
Fix $q:=2+\kappa>2$. Since $\Omega'$ is bounded and $\mu(\Omega')<\infty$, condition
$(\mathcal M_\kappa)$ (which for $N=2$ reads $\mu(B_r(x))\le\Lambda_0r^{\kappa}$) implies, for every
$0<\kappa'\le\kappa$,
\[
\mu(B_r(x))\le\Lambda_0'\,r^{\kappa'}\qquad\forall x\in\R^2,\ r>0 ,
\]
with $\Lambda_0'=\max\{\Lambda_0,\mu(\Omega')\}$: indeed $r^{\kappa}\le r^{\kappa'}$ for $r\le1$,
while for $r>1$ one simply uses $\mu(B_r)\le\mu(\Omega')\le\mu(\Omega')r^{\kappa'}$. Choose
$s\in(1,2)$ so
close to $2$ that
\[
\frac{q(2-s)}{s}\ \le\ \kappa .
\]
Then $\mu(B_r(x))\le\Lambda_0'r^{\frac{q(N-s)}{s}}$ with $N=2$, and $q>2>s$, so Theorem
\ref{thm:adams} (applied with $\alpha=1$ and exponent $s$ in place of $2$; note $\alpha s=s<2=N$)
gives $\|I_1g\|_{L^q(\mu)}\le A\|g\|_{L^s(\R^2)}$. Consequently, for $w\in C_c^\infty(\Omega')$,
\[
\|w\|_{L^q(\Omega',\mu)}\le \tfrac{A}{2\omega_2}\|Dw\|_{L^{s}(\Omega')}
\le C(s,\Omega')\|Dw\|_{L^2(\Omega')}\le C\lambda_0^{-1/2}\|Dw\|_{L^2(\Omega',\mu)} ,
\]
by H\"older's inequality on the bounded set $\Omega'$. Inequality \eqref{eq:sobolev-mean} is
obtained in the same way from \cite[Lemma 7.16]{GT}.
\end{proof}

\begin{remark}
It is instructive to compare \eqref{eq:sobolev} with the corresponding statements in the
literature. In \cite[Theorem 2.2]{DS2}, \cite[Theorem 4.1]{Sci2} and \cite[Theorem 2.2]{MMS} one
finds a weighted Sobolev inequality $\|w\|_{L^{q}(\Omega')}\le C\|Dw\|_{L^2(\Omega',\mu)}$ with the
\emph{Lebesgue} measure on the left; that inequality is trivially true for $1<p<2$ (because
$\rho\ge\lambda_0$), and it is precisely this triviality which is misleading: the Lebesgue norm on
the left cannot absorb the term $\int\rho\,w^2|D\eta|^2$ produced on the right by the Caccioppoli
inequality. Estimate \eqref{eq:sobolev} carries the \emph{same} measure $\mu$ on both sides, and this
is what makes the iteration self-contained.
\end{remark}

\subsection{The weak Harnack comparison inequality}

We now prove Theorem \ref{thm:abstract}. Set $w:=v-u\ge0$ and, for $\varsigma>0$,
$w_\varsigma:=w+\varsigma$. The starting point is the Caccioppoli inequality, which is completely
standard and which we recall for the reader's convenience.

\begin{lemma}\label{lem:cacc-comp}
Let $u,v\in C^1(\Omega)$ satisfy \eqref{eq:scp-hyp} in $B_{6\delta}(x)\subset\Omega'$, set
$\rho=(|Du|+|Dv|)^{p-2}$ and let $\eta\in C_c^1(B_{6\delta}(x))$, $\eta\ge0$. Then for every $\beta<0$,
$\beta\ne-1$, setting $\widetilde w=w_\varsigma^{\frac{\beta+1}{2}}$ and $r=\beta+1$,
\begin{equation}\label{eq:cacc-comp}
\int\rho\,\eta^2|D\widetilde w|^2\,dx\ \le\ \dot C\,r^2\int \widetilde w^{\,2}\big(\eta^2+\rho|D\eta|^2\big)\,dx ,
\end{equation}
while for $\beta=-1$, i.e. $\widetilde w=\log w_\varsigma$,
\begin{equation}\label{eq:cacc-log}
\int\rho\,\eta^2|D\widetilde w|^2\,dx\ \le\ \dot C\int\big(\eta^2+\rho|D\eta|^2\big)\,dx ,
\end{equation}
with $\dot C=\dot C(p,\Lambda,\|v\|_{L^\infty},\|Du\|_{L^\infty},\|Dv\|_{L^\infty})$ independent of
$\varsigma$ and of $\beta$ as long as $|\beta|\ge c>0$.
\end{lemma}

\begin{proof}
This is \cite[(A.1)--(A.9)]{DS2} and \cite[(3.4)--(3.15)]{MMS}; we recall the computation, both
because it is short and because we shall need its exact shape. Hypothesis \eqref{eq:scp-hyp} means
that, for every nonnegative test function $\phi$,
\[
\int_\Omega\big(|Du|^{p-2}Du-|Dv|^{p-2}Dv,\ D\phi\big)dx\ \le\ \Lambda\int_\Omega(v-u)\phi\,dx .
\]
Take $\phi=\eta^2w_\varsigma^{\beta}$ with $\beta<0$, which is admissible since
$w_\varsigma\ge\varsigma>0$ and $w_\varsigma\in C^1$. Writing
$F:=|Dv|^{p-2}Dv-|Du|^{p-2}Du$ and
$D\phi=2\eta w_\varsigma^{\beta}D\eta+\beta\eta^2w_\varsigma^{\beta-1}Dw_\varsigma$, and using
$Dw_\varsigma=Dv-Du$ together with \eqref{eq:vector}, we get
$\beta\,(-F,Dw_\varsigma)\ge|\beta|\hat C\rho|Dw_\varsigma|^2$, whence
\[
|\beta|\hat C\int\eta^2w_\varsigma^{\beta-1}\rho|Dw_\varsigma|^2\,dx
\ \le\ 2\int\eta w_\varsigma^{\beta}|F||D\eta|\,dx+|\Lambda|\int\eta^2w_\varsigma^{\beta+1}dx ,
\]
where we bounded $\Lambda\int(v-u)\phi\le|\Lambda|\int\eta^2w_\varsigma^{\beta+1}$ using
$0\le v-u\le w_\varsigma$. By \eqref{eq:vector2}, $|F|\le\check C\rho|Dw_\varsigma|$, and Young's
inequality applied to
\[
2\check C\int\big(\eta w_\varsigma^{\frac{\beta-1}{2}}\rho^{\frac12}|Dw_\varsigma|\big)
\big(w_\varsigma^{\frac{\beta+1}{2}}\rho^{\frac12}|D\eta|\big)dx
\]
allows us to absorb half of the left-hand side, giving
\[
\int\eta^2w_\varsigma^{\beta-1}\rho|Dw_\varsigma|^2dx\ \le\
\frac{C}{|\beta|}\Big(1+\frac1{|\beta|}\Big)\int w_\varsigma^{\beta+1}\big(\eta^2+\rho|D\eta|^2\big)dx .
\]
Since $|D\widetilde w|^2=\frac{r^2}{4}w_\varsigma^{\beta-1}|Dw_\varsigma|^2$ and
$\widetilde w^{\,2}=w_\varsigma^{\beta+1}$ when $\beta\ne-1$, this is \eqref{eq:cacc-comp}; the
factor $\frac{1}{|\beta|}(1+\frac1{|\beta|})$ is bounded as long as $|\beta|\ge c>0$ and has been
absorbed into $\dot C$. For $\beta=-1$ one has $w_\varsigma^{\beta+1}\equiv1$ and
$|D\widetilde w|^2=w_\varsigma^{-2}|Dw_\varsigma|^2$, which gives \eqref{eq:cacc-log}.
\end{proof}

\begin{proof}[Proof of Theorem \ref{thm:abstract}]
Let $\rho=(|Du|+|Dv|)^{p-2}$ satisfy $(\mathcal M_\kappa)$ in $\Omega'$ and let $q>2$ and $C_S$ be
as in Theorem \ref{thm:sobolev}. Set $\chi:=q/2>1$.

\emph{Step 1: the basic iteration inequality in $\mu$.} Since $\rho\ge\lambda_0$ we have, for any
$\eta\in C_c^1$ with $\|D\eta\|_\infty\le K$,
\begin{equation}\label{eq:passtomu}
\int \widetilde w^{2}\big(\eta^2+\rho|D\eta|^2\big)dx
\le \big(\lambda_0^{-1}+1\big)\int \widetilde w^{2}\big(\eta+|D\eta|\big)^2\rho\,dx
=\big(\lambda_0^{-1}+1\big)\big\|\widetilde w(\eta+|D\eta|)\big\|_{L^2(\mu)}^2 .
\end{equation}
On the other hand, $\eta\widetilde w\in H^{1,2}_{0,\rho}$ and, by Theorem \ref{thm:sobolev},
\[
\|\eta\widetilde w\|_{L^{q}(\mu)}^2\le C_S^2\int\rho|D(\eta\widetilde w)|^2dx
\le 2C_S^2\Big[\int\rho\eta^2|D\widetilde w|^2+\int\rho\widetilde w^2|D\eta|^2\Big].
\]
Combining with \eqref{eq:cacc-comp} and \eqref{eq:passtomu} we obtain the fundamental inequality
\begin{equation}\label{eq:fundamental}
\|\eta\widetilde w\|_{L^{q}(\mu)}^2\ \le\ \ddot C\,(1+|r|)^2\,\big\|\widetilde w\,(\eta+|D\eta|)\big\|_{L^2(\mu)}^2 ,
\end{equation}
where the summand $1$ in $(1+|r|)^2$ accounts for the term $\int\rho\widetilde w^2|D\eta|^2$, which
does not carry the factor $r^2$ produced by \eqref{eq:cacc-comp}.
Analogously, from \eqref{eq:cacc-log},
\begin{equation}\label{eq:fundamental-log}
\int\rho|D\widetilde w|^2\eta^2dx\le \ddot C\int(\eta^2+\rho|D\eta|^2)dx\le \ddot C\,\mu(\Omega')\big(\lambda_0^{-1}+K^2\big) .
\end{equation}

\emph{Step 2: the reverse-H\"older chain.} Define, for $s\ne0$ and $R>0$,
\[
\Phi(s,R,z):=\Big(\int_{B_R(x)}|z|^{s}\,d\mu\Big)^{1/s},\qquad
\Phi(+\infty,R,z):=\sup_{B_R(x)}|z|,\qquad \Phi(-\infty,R,z):=\inf_{B_R(x)}|z| .
\]
Since $\lambda_0\,dx\le d\mu$ and $\mu(\Omega')<\infty$, the measure $\mu$ is finite and mutually
absolutely continuous with respect to $dx$; hence the two limits
$\lim_{s\to\pm\infty}\Phi(s,R,z)=\Phi(\pm\infty,R,z)$ hold, and H\"older's inequality reads
\begin{equation}\label{eq:holderPhi}
\Phi(s_1,R,z)\le\mu(B_R(x))^{\frac1{s_1}-\frac1{s_2}}\,\Phi(s_2,R,z),\qquad 0<s_1<s_2 .
\end{equation}

Fix $\delta\le h'<h''\le5\delta$ and let $\eta\in C_c^1(B_{h''}(x))$ with $\eta\equiv1$ on
$B_{h'}(x)$, $0\le\eta\le1$ and $|D\eta|\le\frac{2}{h''-h'}$. Since $h''-h'\le5\delta$, we have
$\eta+|D\eta|\le\frac{C(\delta)}{h''-h'}$. With $\beta<0$, $\beta\ne-1$, $r:=\beta+1<1$ and
$\widetilde w=w_\varsigma^{r/2}$ we have $\widetilde w^{\,2}=w_\varsigma^{r}$, so that
$\|\widetilde w\|_{L^2(B_{h''}(x),\mu)}^2=\Phi(r,h'',w_\varsigma)^{r}$ and
\[
\|\eta\widetilde w\|^2_{L^{q}(\mu)}
\ \ge\ \Big(\int_{B_{h'}(x)}w_\varsigma^{\chi r}\,d\mu\Big)^{2/q}
=\Phi(\chi r,h',w_\varsigma)^{r} ,
\]
because $\frac{q r}{2}=\chi r$ and $\frac2q=\frac1\chi$. Therefore \eqref{eq:fundamental} becomes
\begin{equation}\label{eq:chain}
\Phi(\chi r,h',w_\varsigma)^{r}\ \le\ \Big(\frac{C(1+|r|)}{h''-h'}\Big)^{2}\,\Phi(r,h'',w_\varsigma)^{r},
\end{equation}
with $C=C(p,\Lambda,\delta,q,\lambda_0,\Lambda_0,\|v\|_{L^\infty},\|Du\|_{L^\infty},\|Dv\|_{L^\infty})$.
Raising \eqref{eq:chain} to the power $1/r$ and remembering that the inequality is reversed when
$r<0$, we obtain
\begin{equation}\label{eq:chain-pos}
\Phi(\chi r,h',w_\varsigma)\ \le\ \Big(\frac{C(1+|r|)}{h''-h'}\Big)^{2/r}\Phi(r,h'',w_\varsigma)
\qquad (0<r<1),
\end{equation}
\begin{equation}\label{eq:chain-neg}
\Phi(\chi r,h',w_\varsigma)\ \ge\ \Big(\frac{C(1+|r|)}{h''-h'}\Big)^{2/r}\Phi(r,h'',w_\varsigma)
\qquad (r<0).
\end{equation}

\emph{Step 3: iteration towards the infimum.} Fix $r_0>0$ (to be chosen in Step 4) and set
\[
r_k:=-r_0\chi^{k},\qquad h_k:=\delta\Big(1+\tfrac32\,2^{-k}\Big),\qquad k\ge0 ,
\]
so that $r_k\to-\infty$, $h_0=\frac{5\delta}{2}$, $h_k\downarrow\delta$ and
$h_k-h_{k+1}=\frac{3\delta}{2}2^{-(k+1)}$. Applying \eqref{eq:chain-neg} with $r=r_k$,
$h'=h_{k+1}$, $h''=h_k$ gives
\[
\Phi(r_{k+1},h_{k+1},w_\varsigma)\ \ge\
\Big(\frac{2^{k+2}C\,(1+r_0\chi^{k})}{3\delta}\Big)^{-\frac{2}{r_0\chi^{k}}}\Phi(r_k,h_k,w_\varsigma) ,
\]
where we used $|r_k|=r_0\chi^{k}$ and $h_k-h_{k+1}=\frac{3\delta}{2}2^{-(k+1)}$. Since
$1+r_0\chi^{k}\le(1+r_0^{-1})r_0\chi^{k}$, taking logarithms and summing over $k\ge0$ leads to the
series
\[
\sum_{k\ge0}\frac{2}{r_0\chi^{k}}\Big[\log\tfrac{4C(1+r_0^{-1})r_0}{3\delta}+k\log 2+k\log\chi\Big] ,
\]
which converges because $\chi>1$. Hence there is $C_1=C_1(r_0,\chi,\delta,C)>0$, independent of
$k$, such that $\Phi(r_k,h_k,w_\varsigma)\ge C_1\,\Phi(-r_0,\frac{5\delta}{2},w_\varsigma)$ for every
$k$. On the other hand $h_k\ge\delta$ and $r_k<0$, so that
$\Phi(r_k,h_k,w_\varsigma)\le\Phi(r_k,\delta,w_\varsigma)$, and the latter converges to
$\Phi(-\infty,\delta,w_\varsigma)$ as $k\to\infty$. Therefore
\begin{equation}\label{eq:toinf}
\inf_{B_{\delta}(x)}w_\varsigma=\Phi(-\infty,\delta,w_\varsigma)\ \ge\ C_1\,\Phi(-r_0,\tfrac{5\delta}{2},w_\varsigma).
\end{equation}

\emph{Step 4: Trudinger's exponential estimate.} We claim that there exist $r_0>0$ and $C_2>0$,
both independent of $\varsigma$, such that
\begin{equation}\label{eq:trudinger}
\Phi\big(r_0,\tfrac{5\delta}{2},w_\varsigma\big)\ \le\ C_2\,\Phi\big(-r_0,\tfrac{5\delta}{2},w_\varsigma\big).
\end{equation}
Put
\[
k_\varsigma:=\exp\Big(\frac{1}{|B_{5\delta}(x)|}\int_{B_{5\delta}(x)}\log w_\varsigma\,dy\Big),
\qquad \widetilde w:=\log\frac{w_\varsigma}{k_\varsigma},
\]
so that $\widetilde w$ has zero \emph{Lebesgue} mean on $B_{5\delta}(x)$. We stress that
$w_\varsigma$ itself is \emph{not} renormalized, it is the function $v-u+\varsigma$, which is the
one satisfying \eqref{eq:scp-hyp}, and that only the auxiliary function $\widetilde w$ carries the
constant $k_\varsigma$; this is legitimate because
$D\widetilde w=Dw_\varsigma/w_\varsigma$ irrespective of $k_\varsigma$, and because $k_\varsigma$
will cancel identically in the final display.

Choosing in \eqref{eq:cacc-log} a cut-off $\eta$ with $\eta\equiv1$ on $B_{5\delta}(x)$ and
$\supp\eta\subset B_{6\delta}(x)$, this is the only point where the hypothesis
$\overline{B_{6\delta}(x)}\subset\Omega'$ is needed, we get from \eqref{eq:fundamental-log}
\begin{equation}\label{eq:logenergy}
\int_{B_{5\delta}(x)}\rho\,|D\widetilde w|^2dy\ \le\ C_3 ,
\end{equation}
with $C_3$ independent of $\varsigma$. Since $\overline{\widetilde w}_{B_{5\delta}(x)}=0$, the
Poincar\'e--Sobolev inequality \eqref{eq:sobolev-mean} and \eqref{eq:logenergy} give
\begin{equation}\label{eq:logLq}
\Phi\big(q,5\delta,\widetilde w\big)=\|\widetilde w\|_{L^{q}(B_{5\delta}(x),\mu)}\ \le\ C_4 ,
\end{equation}
again with $C_4$ independent of $\varsigma$. This uniformity is what allows us to let
$\varsigma\to0$ at the end.

Next we derive a Caccioppoli inequality for powers of $\widetilde w$. Let $\beta\ge1$ and let
$\eta\in C_c^1(B_{5\delta}(x))$, $0\le\eta\le1$. Using in \eqref{eq:scp-hyp} the (nonnegative) test
function
\[
\phi:=\eta^2\,\frac{\Psi(\widetilde w)}{w_\varsigma},\qquad \Psi(t):=|t|^{\beta}+(2\beta)^{\beta},
\]
whose gradient is
\[
D\phi=\frac{2\eta\,\Psi(\widetilde w)}{w_\varsigma}D\eta
+\frac{\eta^2}{w_\varsigma^{2}}\Big[\beta\operatorname{sgn}(\widetilde w)|\widetilde w|^{\beta-1}-\Psi(\widetilde w)\Big]Dw_\varsigma ,
\]
and recalling $D\widetilde w=Dw_\varsigma/w_\varsigma$, $(F,Dw_\varsigma)\ge\hat C\rho|Dw_\varsigma|^2$
and $|F|\le\check C\rho|Dw_\varsigma|$, we obtain
\begin{align*}
&\int\frac{\eta^2}{w_\varsigma^{2}}\Big[\Psi(\widetilde w)-\beta\operatorname{sgn}(\widetilde w)|\widetilde w|^{\beta-1}\Big](F,Dw_\varsigma)\,dy\\
&\qquad\le\ 2\check C\int\eta\,\Psi(\widetilde w)\,\rho\,|D\widetilde w||D\eta|\,dy
+|\Lambda|\int\eta^2\Psi(\widetilde w)\,dy ,
\end{align*}
where we used $0\le v-u\le w_\varsigma$ in the last term. For $\beta\ge1$, Young's inequality with
exponents $\frac{\beta}{\beta-1}$ and $\beta$ yields
$2\beta|t|^{\beta-1}\le\frac{\beta-1}{\beta}|t|^{\beta}+\frac1\beta(2\beta)^{\beta}\le\Psi(t)$,
whence
\begin{equation}\label{eq:elem}
\Psi(t)-\beta\operatorname{sgn}(t)|t|^{\beta-1}\ \ge\ \Psi(t)-\beta|t|^{\beta-1}\ \ge\ \beta|t|^{\beta-1}\ \ge\ 0 .
\end{equation}
Consequently
\[
\beta\hat C\int\eta^2\rho\,|\widetilde w|^{\beta-1}|D\widetilde w|^2
\le 2\check C\int\eta\rho\,\big(|\widetilde w|^{\beta}+(2\beta)^{\beta}\big)|D\widetilde w||D\eta|
+|\Lambda|\int\eta^2\Psi(\widetilde w) .
\]
We split the first integral on the right. Young's inequality gives
\[
2\check C\int\eta\rho|\widetilde w|^{\beta}|D\widetilde w||D\eta|
\le\frac{\beta\hat C}{2}\int\eta^2\rho|\widetilde w|^{\beta-1}|D\widetilde w|^2
+C\int\rho\,|\widetilde w|^{\beta+1}|D\eta|^2 ,
\]
while, for the second piece, \eqref{eq:logenergy} gives
\[
2\check C(2\beta)^{\beta}\int\eta\rho|D\widetilde w||D\eta|
\le (2\beta)^{\beta}\Big[\check C\int\eta^2\rho|D\widetilde w|^2+\check C\int\rho|D\eta|^2\Big]
\le C(2\beta)^{\beta}+\check C(2\beta)^\beta\int\rho|D\eta|^2 .
\]
Finally $\int\eta^2\Psi(\widetilde w)\le C\int\eta^2\big(|\widetilde w|^{\beta+1}+(2\beta)^{\beta}\big)$
because $|t|^{\beta}\le|t|^{\beta+1}+1$ and $1\le(2\beta)^{\beta}$, and, since
$\int\eta^2 dy\ge|B_{5\delta/2}(x)|$ for all the cut-offs used below, the additive constant
$C(2\beta)^{\beta}$ can be written as $C(\delta)\int\eta^2(2\beta)^{\beta}dy$. Absorbing and
collecting terms we arrive at
\begin{equation}\label{eq:cacc-log-beta}
\int\rho\,\eta^2|\widetilde w|^{\beta-1}|D\widetilde w|^2dy\ \le\
C\int\big(|\widetilde w|^{\beta+1}+(2\beta)^{\beta}\big)\big(\eta^2+\rho|D\eta|^2\big)dy ,
\end{equation}
which is the analogue of \eqref{eq:cacc-comp} for the logarithm, with the extra additive term
$(2\beta)^{\beta}$.

Now set $r:=\beta+1\ge2$ and $\widehat w:=|\widetilde w|^{r/2}$, so that
$|D\widehat w|^2=\frac{r^2}{4}|\widetilde w|^{\beta-1}|D\widetilde w|^2$. Exactly as in Step 1,
combining \eqref{eq:cacc-log-beta} with Theorem \ref{thm:sobolev} and \eqref{eq:passtomu} we get,
for $\frac{5\delta}{2}\le h'<h''\le5\delta$,
\[
\Phi(\chi r,h',\widetilde w)^{r}\ \le\ \Big(\frac{C(1+|r|)}{h''-h'}\Big)^{2}
\Big[\Phi(r,h'',\widetilde w)^{r}+(2\beta)^{\beta}\mu(B_{5\delta}(x))\Big] .
\]
Note that here $r\ge2$, so that $1+|r|\le\frac32 r$.
Taking the power $1/r\le\frac12$ and using $(a+b)^{1/r}\le a^{1/r}+b^{1/r}$ together with
$\big((2\beta)^{\beta}\mu(B_{5\delta}(x))\big)^{1/r}\le(2\beta)^{\frac{\beta}{\beta+1}}\max\{1,\mu(B_{5\delta}(x))\}
\le 2r\max\{1,\mu(B_{5\delta}(x))\}$, we obtain
\begin{equation}\label{eq:chain-log}
\Phi(\chi r,h',\widetilde w)\ \le\ \Big(\frac{C|r|}{h''-h'}\Big)^{2/r}
\Big[\Phi(r,h'',\widetilde w)+\gamma\,r\Big],\qquad r\ge2 ,
\end{equation}
with $\gamma=\gamma(\mu(B_{5\delta}(x)))>0$.

We iterate \eqref{eq:chain-log} along $\mu_k:=\frac{5\delta}{2}(1+2^{-k})$ (so that
$\mu_0=5\delta$, $\mu_k\downarrow\frac{5\delta}{2}$ and $\mu_{k-1}-\mu_k=\frac{5\delta}{2}2^{-k}$)
and $r=\chi^{k-1}q$. Setting $A_k:=\Phi(\chi^{k}q,\mu_k,\widetilde w)$ and
$\theta_k:=\big(\tfrac{2^{k+1}C\chi^{k-1}q}{5\delta}\big)^{\frac{2}{\chi^{k-1}q}}$, inequality
\eqref{eq:chain-log} reads $A_k\le\theta_k\,(A_{k-1}+\gamma\chi^{k-1}q)$. Since
$\sum_{k\ge1}\frac{2}{\chi^{k-1}q}\log\big(\tfrac{2^{k+1}C\chi^{k-1}q}{5\delta}\big)<\infty$, the
infinite product $\Theta:=\prod_{k\ge1}\theta_k$ converges, and an elementary induction gives
\[
A_k\ \le\ \Theta A_0+\gamma\Theta q\sum_{j=1}^{k}\chi^{\,j-1}\ \le\ C\big(A_0+\chi^{k}q\big) ,
\]
with $C$ independent of $k$. Given $m\ge q$, let $k_m:=\min\{h\in\N:\ \chi^{h}q\ge m\}$, so that
$\chi^{k_m}q\le\chi m$; by \eqref{eq:holderPhi},
\begin{equation}\label{eq:momentbound}
\Phi\big(m,\tfrac{5\delta}{2},\widetilde w\big)\ \le\
C\,\Phi\big(\chi^{k_m}q,h_{k_m},\widetilde w\big)\ \le\ C\big(\Phi(q,5\delta,\widetilde w)+m\big)
\ \le\ C_5\,(1+m),
\end{equation}
where in the last step we used \eqref{eq:logLq}; the constant $C_5$ does not depend on $\varsigma$
nor on $m$.

We can now conclude Step 4. Note first that \eqref{eq:momentbound} also holds, after enlarging
$C_5$, for the finitely many integers $1\le k<q$: indeed \eqref{eq:holderPhi} gives
$\Phi(k,\frac{5\delta}{2},\widetilde w)\le\mu(B_{5\delta/2}(x))^{\frac1k-\frac1q}\Phi(q,\frac{5\delta}{2},\widetilde w)$,
which is bounded by \eqref{eq:logLq}. Expanding the exponential and using \eqref{eq:momentbound}
for every $k\ge1$ (the term $k=0$ equals $\mu(B_{5\delta/2}(x))$ and is absorbed into the
constant $C$ below),
\[
\int_{B_{5\delta/2}(x)}e^{r_0|\widetilde w|}d\mu
\le\sum_{k\ge0}\frac{r_0^{k}}{k!}\int_{B_{5\delta/2}(x)}|\widetilde w|^{k}d\mu
=\sum_{k\ge0}\frac{r_0^{k}\,\Phi\big(k,\frac{5\delta}{2},\widetilde w\big)^{k}}{k!}
\le C\sum_{k\ge0}\frac{\big(C_5 r_0\big)^{k}(1+k)^{k}}{k!} .
\]
Since $\frac{(1+k)^{k}}{k!}\le e^{k+1}$ by Stirling's bound $k!\ge k^{k}e^{-k}$, the last series
converges as soon as $r_0<\frac{1}{eC_5}$; we fix such an $r_0$, which depends only on the data.
Hence $\int_{B_{5\delta/2}(x)}e^{r_0|\widetilde w|}d\mu\le C$ and therefore
\[
\int_{B_{5\delta/2}(x)}w_\varsigma^{\,r_0}d\mu\cdot\int_{B_{5\delta/2}(x)}w_\varsigma^{-r_0}d\mu
=k_\varsigma^{\,r_0}\!\int e^{r_0\widetilde w}d\mu\cdot k_\varsigma^{-r_0}\!\int e^{-r_0\widetilde w}d\mu
\le\Big(\int e^{r_0|\widetilde w|}d\mu\Big)^{2}\le C^{2} ,
\]
the constant $k_\varsigma$ cancelling exactly, as announced. Raising to the power $\frac1{r_0}$
gives \eqref{eq:trudinger}, with a constant $C_2=C^{2/r_0}$ independent of $\varsigma$.

\emph{Step 5: conclusion.} Let $0<s<\chi$. If $s\le r_0$, then by \eqref{eq:holderPhi}
\[
\Phi(s,2\delta,w_\varsigma)\le\Phi\big(s,\tfrac{5\delta}{2},w_\varsigma\big)
\le\mu\big(B_{5\delta/2}(x)\big)^{\frac1s-\frac1{r_0}}\Phi\big(r_0,\tfrac{5\delta}{2},w_\varsigma\big).
\]
If instead $r_0<s<\chi$, choose $k_0\in\N$ so large that $r_1:=s\chi^{-(k_0+1)}\le r_0$ and set
$\hat r_k:=r_1\chi^{k}$ and $\frac{5\delta}{2}=\varrho_0>\varrho_1>\dots>\varrho_{k_0+1}=2\delta$. Since
$\hat r_{k}\le r_1\chi^{k_0}=s/\chi<1$ for $k\le k_0$, inequality \eqref{eq:chain-pos} may be applied at
each of the $k_0+1$ steps, and it yields
$\Phi(s,2\delta,w_\varsigma)\le C\,\Phi\big(r_1,\frac{5\delta}{2},w_\varsigma\big)$; here $C$ is the
product of the $k_0+1$ factors appearing in \eqref{eq:chain-pos}, hence a finite constant depending
on $s$, on $\delta$ and on the data, but not on $\varsigma$ (this is the only place where the
constant of Theorem \ref{thm:abstract} depends on $s$). By
\eqref{eq:holderPhi} and \eqref{eq:trudinger},
\[
\Phi\big(r_1,\tfrac{5\delta}{2},w_\varsigma\big)\le C\,\Phi\big(r_0,\tfrac{5\delta}{2},w_\varsigma\big)
\le C\,\Phi\big(-r_0,\tfrac{5\delta}{2},w_\varsigma\big) .
\]
In both cases, combining with \eqref{eq:toinf} we obtain
\[
\Phi(s,2\delta,w_\varsigma)\ \le\ C\,\Phi\big(-r_0,\tfrac{5\delta}{2},w_\varsigma\big)
\ \le\ C\,\inf_{B_{\delta}(x)}w_\varsigma ,
\]
with $C$ independent of $\varsigma$. Letting $\varsigma\downarrow0$ and using monotone convergence
on the left and $\inf_{B_{\delta}(x)}w_\varsigma=\varsigma+\inf_{B_{\delta}(x)}w$ on the right, we
conclude that
\[
\Big(\int_{B_{2\delta}(x)}(v-u)^{s}\,d\mu\Big)^{1/s}\ \le\ C\,\inf_{B_{\delta}(x)}(v-u)
\qquad\text{for every }0<s<\chi .
\]

\emph{Step 6: the strong comparison principle.} Let $\Omega_0\subseteq\Omega'$ be a connected
subdomain and set $U:=\{y\in\Omega_0:\ u(y)=v(y)\}$. Since $u,v$ are continuous, $U$ is closed in
$\Omega_0$. If $y_0\in U$, choose $\delta>0$ with $\overline{B_{6\delta}(y_0)}\subset\Omega_0$; then
$\inf_{B_{\delta}(y_0)}(v-u)=0$, hence $\int_{B_{2\delta}(y_0)}(v-u)^s\,d\mu=0$ and, since $\mu$ and $dx$
are mutually absolutely continuous, $u\equiv v$ on $B_{2\delta}(y_0)$. Thus $U$ is open, and the
connectedness of $\Omega_0$ gives the alternative.
\end{proof}

\begin{remark}
Since $\lambda_0\,dx\le d\mu$, the inequality of Theorem \ref{thm:abstract} implies the usual one
with respect to the Lebesgue measure,
$\|v-u\|_{L^s(B_{2\delta}(x))}\le\lambda_0^{-1/s}C\inf_{B_{\delta}(x)}(v-u)$.
\end{remark}

\subsection{The linearized operator}\label{subsec:linearized}
The proof of Theorem \ref{thm:main1} is identical, once the Caccioppoli inequality
\eqref{eq:cacc-comp} is replaced by the following one, which is even simpler because the operator
is linear.

\begin{lemma}\label{lem:cacc-lin}
Let $u\in C^1(\Omega)$ be a weak solution of \eqref{eq:main}, set $\rho=|Du|^{p-2}$ and let
$v\in H^{1,2}_{\rho}(\Omega')\cap L^\infty(\Omega')$, $v\ge0$, be a weak supersolution of
$L_u(v,\cdot)=0$ in $B_{6\delta}(x)\subset\Omega'$. Then, with $v_\varsigma:=v+\varsigma$
$(\varsigma>0)$, $\eta\in C_c^1(B_{6\delta}(x))$, $\eta\ge0$, $\beta<0$, $r=\beta+1$ and
$\widetilde v=v_\varsigma^{r/2}$ if $\beta\ne-1$, $\widetilde v=\log v_\varsigma$ if $\beta=-1$,
the estimates \eqref{eq:cacc-comp} and \eqref{eq:cacc-log} hold with $\widetilde w$ replaced by
$\widetilde v$.
\end{lemma}

\begin{proof}
Write $L_u(v,\varphi)\ge0$ for $\varphi\ge0$ as
$\int(A(x)Dv,D\varphi)\ge\int f'(u)v\varphi$, where
$A(x)=|Du|^{p-2}\big(\mathrm{Id}+(p-2)\frac{Du\otimes Du}{|Du|^{2}}\big)$ satisfies, by
\eqref{eq:ellipticity},
\[
(A\xi,\xi)\ \ge\ \min\{1,p-1\}\,\rho\,|\xi|^2,\qquad |A\xi|\ \le\ (1+|p-2|)\,\rho\,|\xi| .
\]
Take $\varphi=\eta^2v_\varsigma^{\beta}$ with $\beta<0$ and use $Dv_\varsigma=Dv$,
$0\le v\le v_\varsigma$ and $|f'(u)|\le c_1$ on $\supp\eta$ to get
\[
\min\{1,p-1\}\,|\beta|\int\eta^2v_\varsigma^{\beta-1}\rho|Dv_\varsigma|^2
\le 2(1+|p-2|)\int\eta v_\varsigma^{\beta}\rho|Dv_\varsigma||D\eta|
+c_1\int\eta^2v_\varsigma^{\beta+1} .
\]
Young's inequality now gives the claim exactly as in the proof of Lemma \ref{lem:cacc-comp}.
\end{proof}

With Lemma \ref{lem:cacc-lin} in hand, Steps 1--5 of the proof of Theorem \ref{thm:abstract} apply
word for word, the only structural inputs being $(\mathcal M_\kappa)$ through Theorem
\ref{thm:sobolev}, and the two displayed bounds on $A$, and yield
\eqref{eq:weakharnack1}. Analogously, testing with $\eta^2v_\varsigma^{\beta}$, $\beta>0$, in the
reversed inequality produces the same Caccioppoli estimate \eqref{eq:cacc-comp}, now for
$r=\beta+1>1$; since the passage from \eqref{eq:fundamental} to \eqref{eq:chain-pos} only uses
$r>0$, the chain \eqref{eq:chain-pos} is available for these values of $r$ as well, and iterating it
along $r_k=s\chi^{k}\to+\infty$, $h_k=\delta(1+2^{-k})$ gives the local boundedness estimate
\[
\sup_{B_{\delta}(x)}v\ \le\ C\,\|v\|_{L^{s}(B_{2\delta}(x),\mu)}\qquad (s>0)
\]
for nonnegative bounded subsolutions; no logarithmic step is needed here. Combining the two, one
obtains the full Harnack inequality $\sup_{B_{\delta}(x)}v\le C\inf_{B_{2\delta}(x)}v$ for solutions of
$L_u(v,\cdot)=0$, and likewise the Harnack comparison inequality
$\sup_{B_{\delta}(x)}(v-u)\le C\inf_{B_{2\delta}(x)}(v-u)$ of \cite[Corollary 3.2]{DS2}, in the range
$\frac32<p<2$.

%%%%%%%%%%%%%%%%%%%%%%%%%%%%%%%%%%%%%%%%%%%%%%%%%%%%%%%%%%%%%%%%%%%%%%%%%%%%%%%
\section{Proof of the main results}\label{sec:proofs}
%%%%%%%%%%%%%%%%%%%%%%%%%%%%%%%%%%%%%%%%%%%%%%%%%%%%%%%%%%%%%%%%%%%%%%%%%%%%%%%

\begin{proposition}\label{prop:verify}
Let $\frac32<p<2$, let $f$ satisfy $(F)$ and let $u\in C^1(\Omega)$ be a weak solution of
\eqref{eq:main}. Let $\Omega''\Subset\Omega'\Subset\Omega$. Then the weight $\rho=|Du|^{p-2}$
satisfies condition $(\mathcal M_\kappa)$ in $\Omega''$ with
\[
\lambda_0=\|Du\|_{L^\infty(\Omega')}^{\,p-2},\qquad \kappa=\min\{1,\ \tau(3p-4)\}>0 ,
\]
$\tau$ being the H\"older exponent of $Du$ on $\overline{\Omega'}$. The same holds for
$\rho=(|Du|+|Dv|)^{p-2}$ for any $v\in C^1(\Omega)$.
\end{proposition}

\begin{proof}
Set $\sigma:=2-p$. Since $p>\frac32$ we have $\sigma=2-p<p-1$, so Theorem \ref{thm:morrey} applies
and yields $\bar r>0$, $C>0$ and $\kappa=\min\{1,\tau(p-\beta)\}$ with
$\beta=\max\{0,\sigma+2-p\}=4-2p\in(0,1)$, hence $p-\beta=3p-4>0$, such that
\begin{equation}\label{eq:mu-small}
\mu(B_r(x_0))=\int_{B_r(x_0)}|Du|^{p-2}dy\le C\,r^{N-2+\kappa}\qquad\forall x_0\in\Omega'',\ 0<r<\bar r .
\end{equation}
It remains to upgrade \eqref{eq:mu-small} to all centres $x\in\R^N$ and all radii $r>0$, for the
measure $\mu=|Du|^{p-2}\chi_{\Omega''}dx$. If $B_r(x)\cap\Omega''=\emptyset$ there is nothing to
prove. If $B_r(x)\cap\Omega''\ne\emptyset$ and $r<\bar r/2$, pick $x_0\in B_r(x)\cap\Omega''$; then
$B_r(x)\subset B_{2r}(x_0)$ and \eqref{eq:mu-small} gives
$\mu(B_r(x))\le C(2r)^{N-2+\kappa}$. If $r\ge\bar r/2$ then, taking $r=\bar r/2$ in the previous case
and covering $\Omega''$ by finitely many such balls, $\mu(\Omega'')<\infty$ and
\[
\mu(B_r(x))\le\mu(\Omega'')\le\mu(\Omega'')\Big(\frac{2}{\bar r}\Big)^{N-2+\kappa}r^{N-2+\kappa} .
\]
Finally $\rho\ge\lambda_0$ by \eqref{eq:lowerweight}, and
$(|Du|+|Dv|)^{p-2}\le|Du|^{p-2}$ pointwise, so the last assertion follows as well (with
$\lambda_0$ replaced by $(\|Du\|_\infty+\|Dv\|_\infty)^{p-2}$).
\end{proof}

\begin{proof}[Proof of Theorems \ref{thm:main1}, \ref{thm:main2} and \ref{thm:main4}]
Let $\Omega''\Subset\Omega'\Subset\Omega$. By Proposition \ref{prop:verify}, the weight
$\rho=|Du|^{p-2}$ (resp. $(|Du|+|Dv|)^{p-2}$) satisfies $(\mathcal M_\kappa)$ in $\Omega''$.
Theorem \ref{thm:main4} is then exactly Theorem \ref{thm:abstract} applied in $\Omega''$: given a
connected subdomain $\Omega_0\subseteq\Omega$, the set $\{u=v\}\cap\Omega_0$ is closed in $\Omega_0$
by continuity, and it is open because every $y_0\in\Omega_0$ has a ball $\overline{B_{6\delta}(y_0)}$
contained in some $\Omega''\Subset\Omega'\Subset\Omega_0$, to which Theorem \ref{thm:abstract}
applies. The statement with the hypothesis
$-\Lap{p}(u)-f(u)\le-\Lap{p}(v)-f(v)$ reduces to \eqref{eq:scp-hyp} exactly as in
\cite[Remark 3.2]{DS2}: since $f$ is locally Lipschitz on $(0,\infty)$, for every compact
$[a,b]\subset(0,\infty)$ there is $\Lambda>0$ such that $s\mapsto f(s)+\Lambda s$ is nondecreasing on
$[a,b]$.

Theorem \ref{thm:main1} is the weak Harnack inequality obtained by running Steps 1--5 of the proof
of Theorem \ref{thm:abstract} on the linearized equation, as detailed in \S\ref{subsec:linearized};
note that $\rho=|Du|^{p-2}$ satisfies $(\mathcal M_\kappa)$ by Proposition \ref{prop:verify}. In
both Theorem \ref{thm:main1} and Theorem \ref{thm:main4} one may take $\chi=q/2$, with $q$ given by
\eqref{eq:qdef} and $\kappa=\min\{1,\tau(3p-4)\}$ as in Proposition \ref{prop:verify}.

Theorem \ref{thm:main2} follows from Theorem \ref{thm:main1} in the classical way: if $\Omega_0$ is
a connected subdomain of $\Omega$, $v\ge0$ is a continuous supersolution in $\Omega_0$ and
$K:=\{v=0\}\cap\Omega_0$, then $K$ is closed in $\Omega_0$ by continuity and open by
\eqref{eq:weakharnack1} (applied on balls $B_{6\delta}(x)$ contained in $\Omega_0$); hence
$K=\emptyset$ or $K=\Omega_0$. The final assertion follows since
$u_{x_i}$ solves the linearized equation, by \eqref{eq:linear-eq}.
\end{proof}

\begin{proof}[Proof of Theorem \ref{thm:main3}]
By \cite[Theorem 2.3 and Corollary 2.1]{DS2} (proved in \cite{DS1}), which hold for every
$1<p<\infty$ and rely only on the
weak comparison principle and the moving plane method, for every direction $\nu$ and every
$\lambda\in(a(\nu),\lambda_2(\nu))$ one has $u(x)\le u(x_\lambda^\nu)$ in $\Omega^\nu_\lambda$, and
$\frac{\partial u}{\partial\nu}\ge0$ in $\Omega^{\nu}_{\lambda_2(\nu)}$. By Theorem \ref{thm:main2}, in each connected component $\mathcal C$ of $\Omega^{\nu}_{\lambda_2(\nu)}$ either
$\frac{\partial u}{\partial\nu}\equiv0$ or $\frac{\partial u}{\partial\nu}>0$. The first alternative
is impossible: $\mathcal C\cap\partial\Omega\ne\emptyset$, $u=0$ on $\partial\Omega$ and $u>0$ in
$\Omega$, so $u$ cannot be constant along $\nu$ in $\mathcal C$. Hence
$\frac{\partial u}{\partial\nu}>0$ in $\Omega^{\nu}_{\lambda_2(\nu)}$, i.e.
\[
Z_u\cap\Omega^{\nu}_{\lambda_2(\nu)}=\emptyset .
\]
When $\Omega$ is convex and symmetric in the directions $e_1,\dots,e_N$ one has
$\lambda_2(\pm e_i)=0$ for every $i$, so that $Z_u$ is contained in
$\bigcap_{i=1}^N\{x_i=0\}=\{0\}$; on the other hand $Z_u\ne\emptyset$ because $u$ attains its
maximum in $\Omega$. Hence $Z_u=\{0\}$. Since the $p$-Laplace operator is uniformly elliptic with
H\"older continuous coefficients on any compact subset of $\Omega\setminus Z_u$, elliptic
regularity gives $u\in C^2(\Omega\setminus\{0\})$; see \cite[Theorem 1.3]{DS2}.
\end{proof}

\begin{proof}[Proof of Corollary \ref{cor:ball}]
Let $\Omega=B_R(0)$ and let $u$ be a solution of
\eqref{eq:dirichlet}. By \cite[Corollary 2.1]{DS2} (valid for all $1<p<\infty$), $u$ is radially
symmetric and radially decreasing. Writing $u=u(s)$, $s=|x|$, integrating the equation on $B_s(0)$
gives
\[
s^{N-1}|u'(s)|^{p-2}(-u'(s))=\int_0^{s}f(u(t))t^{N-1}dt ,
\]
so that, with $c_*:=\min_{\overline{B_R}}f(u)>0$ and $c^*:=\max_{\overline{B_R}}f(u)$ (cf.
\cite{SerrinTang} for the asymptotics of radial solutions near the origin),
\begin{equation}\label{eq:radial}
\Big(\frac{c_*}{N}\Big)^{\frac1{p-1}}s^{\frac1{p-1}}\ \le\ |Du(x)|\ \le\
\Big(\frac{c^*}{N}\Big)^{\frac1{p-1}}s^{\frac1{p-1}},\qquad s=|x| .
\end{equation}
Let $\Omega''\Subset B_R(0)$, $\sigma=2-p$ and put $\theta:=\frac{2-p}{p-1}$, which satisfies
$\theta<2$ precisely because $p>\frac43$. Let $x_0\in\Omega''$, $d:=|x_0|$ and $r>0$. If
$d\ge2r$ then $|Du|\ge cd^{\frac1{p-1}}\ge c(2r)^{\frac{1}{p-1}}$ on $B_r(x_0)$ by \eqref{eq:radial},
whence
\[
\int_{B_r(x_0)}|Du|^{-\sigma}\le C\,r^{-\theta}r^{N}=C\,r^{N-\theta} .
\]
If $d<2r$ then $B_r(x_0)\subset B_{3r}(0)$ and, by \eqref{eq:radial},
\[
\int_{B_r(x_0)}|Du|^{-\sigma}\le C\int_0^{3r}s^{-\theta}s^{N-1}ds=C\,r^{N-\theta},
\]
the integral being convergent since $\theta<2\le N$. In both cases
$\mu(B_r(x_0))\le Cr^{N-2+(2-\theta)}$ with $2-\theta>0$, and one extends the estimate to all
centres and radii as in the proof of Proposition \ref{prop:verify}. Thus $(\mathcal M_\kappa)$ holds
with $\kappa=2-\theta$, and Theorems \ref{thm:abstract} and \ref{thm:sobolev} apply.
\end{proof}

\begin{remark}\label{rem:minkowski}
The proof of Corollary \ref{cor:ball} uses only the lower bound in \eqref{eq:radial}. More
generally, suppose that $Z_u\cap\Omega''$ has upper Minkowski dimension $d\in[0,N-1]$, in the
quantitative sense that
\[
\big|\{y\in B_r(x_0):\ \dist(y,Z_u)<t\}\big|\ \le\ C\,r^{\,d}\,t^{\,N-d}
\qquad (0<t\le r),
\]
and that $|Du(y)|\ge c\,\dist(y,Z_u)^{\frac1{p-1}}$. The exponent $\frac1{p-1}$ is the maximal
degeneracy rate compatible with \eqref{eq:main}: integrating the equation over $B_s(x_0)$ with
$x_0\in Z_u$ and using the divergence theorem gives
$c_0\,\omega_Ns^{N}\le\int_{B_s(x_0)}f(u)\le N\omega_N s^{N-1}\max_{\partial B_s(x_0)}|Du|^{p-1}$,
whence $\max_{\partial B_s(x_0)}|Du|\ge(c_0s/N)^{\frac1{p-1}}$ for every small $s>0$. Writing $\theta=\frac{2-p}{p-1}$, one then
gets $\mu(B_r(x_0))\lesssim r^{\,N-\theta}$ provided
\[
\theta<\min\{2,\ N-d\} ,
\]
the constraint $\theta<N-d$ being the one guaranteeing $\rho\in L^1_{\rm loc}$ and the constraint
$\theta<2$ the Morrey exponent. Hence $(\mathcal M_\kappa)$ holds for $p>\frac43$ whenever
$d\le N-2$, and for $p>\frac32$ when $d=N-1$. The case $d=N-1$ is realized by the solutions of
Proposition \ref{prop:sharp} and by those of Proposition \ref{prop:annulus}, and the case $d=0$ by
the radial solutions of Corollary
\ref{cor:ball}. This is the precise sense in which the threshold $\frac32$ of Theorems
\ref{thm:main1}--\ref{thm:main3} is dictated by the possible presence of a hypersurface of critical
points.
\end{remark}

%%%%%%%%%%%%%%%%%%%%%%%%%%%%%%%%%%%%%%%%%%%%%%%%%%%%%%%%%%%%%%%%%%%%%%%%%%%%%%%
\section{Applications}\label{sec:applications}
%%%%%%%%%%%%%%%%%%%%%%%%%%%%%%%%%%%%%%%%%%%%%%%%%%%%%%%%%%%%%%%%%%%%%%%%%%%%%%%

The strong comparison principle of \cite{DS2}, together with the strong maximum principle for
the linearized operator, which is its companion, is one of the standard tools in the qualitative
theory of $p$-Laplace equations, and they are used as black boxes in a number of papers, each of
which inherits the hypothesis $\frac{2N+2}{N+2}<p<2$. For the papers discussed below we have
checked that this hypothesis is invoked \emph{only} through \cite[Theorems 1.2 and 1.4]{DS2}, so
that replacing the latter by our Theorems \ref{thm:main2} and \ref{thm:main4} widens the admissible
range at once.

\subsection{A more flexible hypothesis}
In the applications one rarely has $f>0$ on $(0,\infty)$ together with $u>0$; what one really has
is that the \emph{composition} $f(u)$ has a constant sign on the subdomain under consideration.
Accordingly we introduce:

\medskip
\noindent$(\mathcal H)$\quad $u\in C^1(\Omega)$ is a weak solution of $-\Delta_pu=f(u)$ in
$\Omega$, where $f$ is locally Lipschitz continuous on an open interval containing
$u(\Omega)$, and $f(u(x))>0$ for every $x\in\Omega$, or $f(u(x))<0$ for every $x\in\Omega$.
\medskip

\begin{proposition}\label{prop:H}
Theorems \ref{thm:main1}, \ref{thm:main2}, \ref{thm:main4}, \ref{thm:morrey}, Proposition
\ref{prop:verify} and Corollary \ref{cor:summ} remain valid, with the same proofs, if the hypothesis
``$f$ satisfies $(F)$ and $u>0$'' is replaced by $(\mathcal H)$.
\end{proposition}

\begin{proof}
Assume first $f(u)<0$ in $\Omega$ and set $\widetilde u:=-u$, $\widetilde f(s):=-f(-s)$. Then
$|D\widetilde u|=|Du|$, $Z_{\widetilde u}=Z_u$,
$\Delta_p\widetilde u=-\Delta_p u$, so that $-\Delta_p\widetilde u=\widetilde f(\widetilde u)$ with
$\widetilde f(\widetilde u)=-f(u)>0$ in $\Omega$; moreover $u\le v$ if and only if
$\widetilde v\le\widetilde u$, and the linearized operators at $u$ and at $\widetilde u$ coincide.
Hence we may and do assume $f(u)>0$ in $\Omega$.

It remains to observe where the hypotheses on $f$ and $u$ are used. In Section \ref{sec:morrey}
they enter only through
\begin{enumerate}[label=\rm(\roman*),leftmargin=2.2em]
\item the local $C^{1,\tau}$ regularity of Theorem \ref{thm:C1alpha}, which only requires
$f(u)\in L^\infty_{\rm loc}(\Omega)$;
\item the validity of the linearized equation \eqref{eq:linear-eq} and of the second-order estimate
of Theorem \ref{thm:DSlin}; these are established in \cite[Theorem 1.1 and (1.2)]{DS1}, and an
inspection of the proofs given there shows that the only property of $f$ which is used is its local
Lipschitz continuity on an interval containing $u(\overline{\Omega'})$, this is also the way in
which these results are applied in \cite{EFMS1,EFMS2}, where the solutions are not assumed to be
positive;
\item the two constants $c_0=\min_{\overline{\Omega'}}f(u)>0$ and
$c_1=\sup_{\overline{\Omega'}}|f'(u)|<\infty$, whose finiteness and positivity are exactly what
$(\mathcal H)$ provides on compact subsets.
\end{enumerate}
Sections \ref{sec:sobolev} and \ref{sec:proofs} use no property of $f$ beyond the local Lipschitz
continuity and the conclusion of Theorem \ref{thm:morrey}.
\end{proof}

\begin{remark}
Hypothesis $(\mathcal H)$ is precisely the assumption $(f_u)$--$(f_v)$ under which the strong
comparison principle is quoted in \cite[Theorem 2.4]{EFMS1} and \cite[Theorem 2.2]{EFMS2}, and it is
what makes the results below directly applicable to nonlinearities that change sign globally, as
long as they have a sign on the subdomain where the comparison is performed.
\end{remark}

\subsection{Gibbons' conjecture for the \texorpdfstring{$p$}{p}-Laplacian}
In \cite{EFMS1}, Esposito, Farina, Montoro and Sciunzi proved Gibbons' conjecture for
$-\Delta_pu=f(u)$ in $\R^N$ under the assumptions
\begin{itemize}[leftmargin=2.6em]
\item[$(h_f)$] $f\in C^1([-1,1])$, $f(-1)=f(1)=0$, $f'_+(-1)<0$, $f'_-(1)<0$, and
$\mathcal N_f:=\{t\in[-1,1]:\ f(t)=0\}$ is a finite set.
\end{itemize}
As far as we can see, the hypothesis $\frac{2N+2}{N+2}<p<2$ is invoked there only through
\cite[Theorems 2.4 and 2.6]{EFMS1}, which are verbatim \cite[Theorems 1.4 and 1.2]{DS2}, i.e. the
strong comparison principle and the strong maximum principle for the linearized operator, and in
the accompanying remark that $\rho=|Du|^{p-2}\in L^1$, which also holds precisely for
$p>\frac32$; the weak comparison principle in half-spaces proved in \cite[\S3]{EFMS1}, the other
main ingredient of that paper, carries no such restriction. Replacing those two theorems by our
Theorems \ref{thm:main4} and \ref{thm:main2} (in the form given by Proposition \ref{prop:H}) one
obtains the conclusion of \cite[Theorem 1.1]{EFMS1} for every $\frac32<p<2$.

In fact one can do better, and weaken $(h_f)$ at the same time. In \cite{LeFM} the author proved
Gibbons' conjecture, by a variant of the sliding method which avoids the linearized equation
altogether, under the following condition, in which the differentiability of $f$ at $\pm1$ is not
required:
\begin{itemize}[leftmargin=2.6em]
\item[$(\tilde h_f)$] $f:[-1,1]\to\R$ is Lipschitz continuous, $f(-1)=f(1)=0$, the set
$Z_f:=\{z\in(-1,1):\ f(z)=0\}$ is finite, and there exists $0<\delta<1$ such that $f$ is strictly
decreasing on $[-1,-1+\delta]$ and on $[1-\delta,1]$.
\end{itemize}
(Since $f$ is continuous and the two intervals are compact, $(\tilde h_f)$ is equivalent to the
quantitative form in which it is stated in \cite{LeFM}: for every $0<\eps<\delta$ the supremum of the
difference quotients $\frac{f(t_2)-f(t_1)}{t_2-t_1}$ over $t_2-t_1\ge\eps$ is negative.) The argument of \cite{LeFM} uses the restriction on $p$ only through the strong
comparison principle of \cite[Theorem 1.2]{MMS} and the strong maximum principle for the linearized
operator of \cite[Theorem 1.4]{Mon}, which are quoted there as Theorems 5 and 6; both are improved
to the range $\frac32<p<2$ by Theorem \ref{thm:firstorder} below (and, in the case $a\equiv0$,
already by Theorems \ref{thm:main4} and \ref{thm:main2}). We therefore obtain:

\begin{theorem}\label{thm:gibbons}
Let $N>1$ and $\frac32<p<2$, and let $f$ satisfy $(\tilde h_f)$. Let $u\in C^1(\R^N)$ be a weak
solution of $-\Delta_pu=f(u)$ in $\R^N$ with $|u|\le1$ and
\[
\lim_{y\to+\infty}u(x',y)=1,\qquad \lim_{y\to-\infty}u(x',y)=-1
\]
uniformly with respect to $x'\in\R^{N-1}$. Then $u$ depends only on $y$ and $\partial_yu>0$ in
$\R^N$.
\end{theorem}

\begin{remark}
Condition $(\tilde h_f)$ is strictly weaker than $(h_f)$. Indeed, if $f\in C^1$ near $\pm1$ with
$f'<0$ on $(-1,-1+\delta]\cup[1-\delta,1)$, then $(\tilde h_f)$ holds
\cite[Remark 2]{LeFM}; in particular $(h_f)$ implies $(\tilde h_f)$, whereas $(\tilde h_f)$ requires
neither the differentiability of $f$ at $\pm1$ nor, when the one-sided derivatives exist, that
$f'_+(-1)$ and $f'_-(1)$ be different from $0$. The assumption
$f'_+(-1)<0$, $f'_-(1)<0$ is essential in the approach of \cite{EFMS1}, where it is used to derive a
key estimate for the linearized equation, and this is precisely what the sliding method of
\cite{LeFM} dispenses with. A typical nonlinearity covered by $(\tilde h_f)$ but not by $(h_f)$ is
$f(u)=|u|^{r}u\,|1-u^2|^{s}(1-u^2)$ with $r,s>0$.
\end{remark}

\subsection{Monotonicity in half-spaces for changing-sign nonlinearities}
In \cite{EFMS2} the same authors proved the monotonicity, in the direction orthogonal to the
boundary, of positive solutions of $-\Delta_pu=f(u)$ in a half-space under zero Dirichlet boundary
conditions, for a genuinely changing-sign nonlinearity, under the assumptions
\begin{itemize}[leftmargin=2.6em]
\item[$(h_f')$] $f\in C^1([0,+\infty))$, $f(0)\ge0$, and
$\mathcal N_f:=\{t\in[0,+\infty):\ f(t)=0\}$ is a discrete set.
\end{itemize}
Again, the restriction $\frac{2N+2}{N+2}<p<2$ enters only through \cite[Theorems 2.2 and
2.3]{EFMS2}, which are \cite[Theorems 1.4 and 1.2]{DS2}, and through the requirement
$\rho\in L^1$. Consequently:

\begin{theorem}\label{thm:halfspace}
Let $N>1$, $\frac32<p<2$, and let $f$ satisfy $(h_f')$. Let
$u\in C^{1,\alpha}_{\rm loc}(\R^N_+)$, $\alpha\in(0,1)$, be a weak solution of
\[
-\Delta_pu=f(u)\ \text{ in }\R^N_+,\qquad u>0\ \text{ in }\R^N_+,\qquad u=0\ \text{ on }\partial\R^N_+,
\]
with $Du\in L^\infty(\Sigma_\lambda)$ for every strip
$\Sigma_\lambda=\R^{N-1}\times(0,\lambda)$. Then $u$ is monotone increasing with respect to the
$y$-direction, $\frac{\partial u}{\partial y}\ge0$ in $\R^N_+$, and
$\frac{\partial u}{\partial y}>0$ in $\R^N_+\setminus Z_{f(u)}$, where
$Z_{f(u)}:=\{x\in\R^N_+:\ u(x)\in\mathcal N_f\}$.
\end{theorem}

\subsection{Equations with a first-order term}
The Harnack comparison inequality and the strong comparison principle of Merch\'an, Montoro and
Sciunzi \cite[Theorems 1.1--1.2]{MMS} concern
\[
-\Delta_pu+a(x,u)|Du|^{q}=f(x,u),\qquad q\ge\max\{p-1,1\},
\]
and are proved by exactly the scheme discussed in this paper; there the restriction
$\frac{2N+2}{N+2}<p<2$ is used at a single point, namely in Case~(b) of the proof of
\cite[Theorem 3.1]{MMS}, in order to guarantee that $\rho\in L^{t}$ with $t>N/2$. The strong
comparison and strong maximum principles of \cite[Theorems 1--3]{Le}, which deal with the same
equation and a sign-changing $f$, inherit the restriction \emph{only} from
\cite[Theorem 1.2]{MMS} and from the strong maximum principle for the linearized operator of
\cite[Theorem 1.4]{Mon}, the latter being obtained by the same scheme. Our Theorem \ref{thm:abstract} replaces that requirement by
$(\mathcal M_\kappa)$, so the exponent can be lowered to $\frac32$ as soon as the analogue of
Theorem \ref{thm:morrey} is available in that setting. The latter follows by the very same
dichotomy: the second-order estimate is \cite[Theorem 2.1]{MMS}, whose proof is carried out with a
cut-off function and therefore delivers at once the localized form of Lemma \ref{lem:cacc2} that we
need; whereas in the proof of Lemma \ref{lem:cacc1} the first-order term contributes, after testing
with $\psi^2g_\eps$, the additional summand
\[
\int_{\Omega'}|a(x,u)|\,|Du|^{q}\psi^2g_\eps\,dx\ \le\ K\int_{\Omega'}\psi^2|Du|^{q-\sigma}dx
\ \le\ K\,S_r^{\,q-\sigma}\,|B_{2r}| ,
\]
which is harmless because $q\ge\max\{p-1,1\}>\sigma$, so that this term is bounded by $Cr^{N}$ in
Case~1 and by $Cr^{N+\tau(q-\sigma)}$ in Case~2 of the proof of Theorem \ref{thm:morrey}; both
exponents exceed $N-2$. We omit the remaining (routine) details and simply record the outcome:

\begin{theorem}\label{thm:firstorder}
The Harnack comparison inequality and the strong comparison principle of
\cite[Theorems 1.1--1.2]{MMS}, the strong maximum principle of \cite[Theorem 1.4]{Mon}, and the
strong comparison and strong maximum principles of \cite[Theorems 1--3]{Le}, all hold in the range
$\frac32<p<2$ in place of $\frac{2N+2}{N+2}<p<2$.
\end{theorem}

\begin{remark}
The list above is not exhaustive: the same substitution can be performed in every result whose only
use of \cite[Theorems 1.2 and 1.4]{DS2} is as a black box. We also stress the converse reading of
Proposition \ref{prop:sharp}: since the weight $|Du|^{p-2}$ need not be locally integrable for
$1<p\le\frac32$, the value $\frac32$ is where all these statements will have to stop, unless the
linearization in $H^{1,2}_{\rho}$ is abandoned.
\end{remark}

%%%%%%%%%%%%%%%%%%%%%%%%%%%%%%%%%%%%%%%%%%%%%%%%%%%%%%%%%%%%%%%%%%%%%%%%%%%%%%%
\section{Optimality of the threshold}\label{sec:remarks}
%%%%%%%%%%%%%%%%%%%%%%%%%%%%%%%%%%%%%%%%%%%%%%%%%%%%%%%%%%%%%%%%%%%%%%%%%%%%%%%

\begin{proof}[Proof of Proposition \ref{prop:sharp}]
Let $1<p<2$, $p'=\frac{p}{p-1}>2$, and let $u$ be as in \eqref{eq:example}. Then
\[
Du(x)=-|x_1|^{\frac{1}{p-1}}\operatorname{sgn}(x_1)\,e_1 ,
\]
which is continuous since $\frac1{p-1}>0$; hence $u\in C^1(\R^N)$ and
$|Du(x)|=|x_1|^{\frac1{p-1}}$. Moreover
\[
|Du|^{p-2}Du=-|x_1|^{\frac{p-2}{p-1}}|x_1|^{\frac{1}{p-1}}\operatorname{sgn}(x_1)e_1=-x_1e_1 ,
\]
so that $\Lap{p}(u)=\operatorname{div}(-x_1e_1)=-1$, i.e. $-\Lap{p}(u)=1$. Thus $u$ solves
\eqref{eq:main} with $f\equiv1$, which satisfies $(F)$, on any ball $B$ where $u>0$; such balls exist
for every $c>0$ and, for $c$ large, $B$ can be taken of any prescribed radius. Clearly
$Z_u=\{x_1=0\}$.

(i) For a ball $B\subset\{|x_1|<a\}$ meeting $\{x_1=0\}$,
\[
\int_B|Du|^{-\sigma}dx=\int_B|x_1|^{-\frac{\sigma}{p-1}}dx ,
\]
which is finite if and only if $\frac{\sigma}{p-1}<1$, i.e. $\sigma<p-1$. Recalling that
\eqref{eq:DSsummability} asserts the finiteness of $\int|Du|^{-(p-1)r}$ for every $r<1$, this shows
that the exponent $p-1$ therein cannot be increased.

(ii) Here $\rho=|Du|^{p-2}=|x_1|^{-\frac{2-p}{p-1}}$, and $\rho\in L^1(B)$ if and only if
$\frac{2-p}{p-1}<1$, i.e. $2-p<p-1$, i.e. $p>\frac32$.

(iii) Let $\theta=\frac{2-p}{p-1}$ and $x_0\in Z_u\cap B$. By Fubini's theorem,
\[
\int_{B_r(x_0)}\rho\,dx\simeq r^{N-1}\int_{-r}^{r}|t|^{-\theta}\,dt ,
\]
which is finite exactly when $\theta<1$, i.e. $p>\frac32$, in agreement with (ii); in that case it
equals $\simeq r^{N-\theta}=r^{\,N-2+(2-\theta)}$ with $2-\theta>1$, so that $(\mathcal M_\kappa)$
holds with $\kappa=2-\theta$. Note that the Morrey requirement alone would only ask $\theta<2$,
i.e. $p>\frac43$; it is the integrability of $\rho$, equivalently $\theta<1$, that is binding
here, because $Z_u$ is a hypersurface.
\end{proof}

\begin{proof}[Proof of Proposition \ref{prop:annulus}]
Existence and uniqueness of $u\in W^{1,p}_0(\Omega)$ follow from the strict monotonicity of
$-\Lap{p}$, and $u\in C^{1,\alpha}(\overline\Omega)$ by \cite{Lie}. Since $\Omega$ and the equation
are invariant under rotations, $u\circ R$ solves \eqref{eq:annulus} for every rotation $R$, so
uniqueness forces $u$ to be radial; we write $u=u(r)$, $r=|x|$.

In radial coordinates \eqref{eq:annulus} reads $-\big(r^{N-1}|u'|^{p-2}u'\big)'=r^{N-1}$, hence
there is a constant, which we write as $r_0^N/N$ with $r_0>0$, such that
\begin{equation}\label{eq:annulus-ode}
\Phi\big(u'(r)\big)=g_{r_0}(r):=\frac{r_0^{N}-r^{N}}{N\,r^{N-1}},\qquad
\Phi(s):=|s|^{p-2}s ,
\end{equation}
so that $u'(r)=\Phi^{-1}(g_{r_0}(r))$ with $\Phi^{-1}(y)=|y|^{\frac{2-p}{p-1}}y$ continuous and
strictly increasing. For fixed $r$ the map $r_0\mapsto g_{r_0}(r)$ is strictly increasing, hence so
is $r_0\mapsto\int_a^b\Phi^{-1}(g_{r_0}(r))\,dr$; this integral is negative for $r_0=a$ (where
$g_a<0$ on $(a,b]$) and positive for $r_0=b$ (where $g_b>0$ on $[a,b)$). Therefore there is exactly
one $r_0\in(a,b)$ for which $\int_a^bu'=0$, i.e. for which $u(a)=u(b)=0$, and this is the value
occurring in \eqref{eq:annulus-ode}.

Since $g_{r_0}$ is strictly decreasing on $(0,\infty)$ and vanishes only at $r_0$, we get $u'>0$ on
$(a,r_0)$ and $u'<0$ on $(r_0,b)$; together with $u(a)=u(b)=0$ this gives $u>0$ in $\Omega$, and
\[
Z_u=\{r=r_0\},\qquad
\inf\big\{|Du(x)|:\ \big||x|-r_0\big|\ge\delta\big\}>0\quad\text{for every }\delta>0 .
\]
Moreover $g_{r_0}(r_0)=0$ and $g_{r_0}'(r_0)=-1$, whence $|g_{r_0}(r)|/|r-r_0|\to1$ as $r\to r_0$
and therefore
\[
\frac{|Du(x)|}{\big||x|-r_0\big|^{\frac1{p-1}}}
=\Big(\frac{|g_{r_0}(r)|}{|r-r_0|}\Big)^{\frac1{p-1}}\longrightarrow 1 .
\]
Finally,
\[
\int_\Omega|Du|^{p-2}dx=N\omega_N\int_a^b|u'(r)|^{p-2}r^{N-1}dr ,
\]
whose integrand is bounded away from $r_0$ and comparable with $|r-r_0|^{-\theta}$ near $r_0$; it is
therefore finite if and only if $\theta<1$, i.e. $p>\frac32$. The local statements at a point
$x_0\in Z_u$ follow exactly as in part (iii) of the proof of Proposition \ref{prop:sharp}, with the
sphere $\{|x|=r_0\}$ in place of the hyperplane $\{x_1=0\}$.
\end{proof}

\begin{remark}[Why $\frac32$ is a genuine barrier]\label{rem:barrier}
Every treatment of the strong comparison principle across the critical set known to us, namely
\cite{DS2,Sci2,MMS,Mon,Le} and the present paper, proceeds by linearizing and working in the
weighted space $H^{1,2}_{\rho}$ with $\rho=|Du|^{p-2}$ (or $(|Du|+|Dv|)^{p-2}$). By Propositions
\ref{prop:sharp}(ii) and \ref{prop:annulus}, when $1<p\le\frac32$ this space is not even defined for
all solutions: the
weight fails to be locally integrable, so that $\mu=\rho\,dx$ is not a Radon measure, no Sobolev
inequality of the form \eqref{eq:sobolev} can hold, and the natural energy
$\int\rho|D\varphi|^2$ assigns infinite mass to constants. Observe that the failure is caused by the
codimension-one geometry of $Z_u$: it is the fact that the critical set can be a hypersurface, along
which $|Du|$ degenerates at the maximal admissible rate $\dist(\cdot,Z_u)^{\frac1{p-1}}$, that
saturates the summability exponent in \eqref{eq:DSsummability}. This is precisely the phenomenon
recorded in \cite[Example 5.2]{Sci2}, and it is explicitly noted in \cite{MMS} that
the weight need not belong to $L^1$ when $p$ is close to $1$. Consequently, the value
$\frac32$ obtained in Theorems \ref{thm:main1}--\ref{thm:main3} is optimal within the
weighted-linearization framework, and any progress below it must come from a genuinely different
idea.

It should be stressed that this obstruction is not confined to the local theory. The solution
\eqref{eq:example} is only a local solution of \eqref{eq:main}; but Proposition \ref{prop:annulus}
exhibits the same failure for the Dirichlet problem \eqref{eq:dirichlet} on a bounded smooth
domain, with the simplest possible nonlinearity $f\equiv1$, the critical set being a smooth
hypersurface along which $|Du|$ degenerates at the maximal admissible rate. One caveat remains:
condition $(\mathcal M_\kappa)$ is a
hypothesis on the particular pair $(u,v)$ under consideration; Theorem \ref{thm:abstract} therefore
applies, for any $p>1$, to every solution whose critical set is small enough in the sense of Remark
\ref{rem:minkowski}.
\end{remark}

\begin{remark}[Some observations on the range $1<p\le\frac32$]
Let $w=v-u\ge0$ and $F:=|Dv|^{p-2}Dv-|Du|^{p-2}Du$, so that $-\operatorname{div}F+\Lambda w\ge0$
weakly. For $1<p<2$ and $Du,Dv$ bounded, \eqref{eq:vector} gives the \emph{uniform} lower ellipticity
\[
(F,Dw)\ \ge\ \hat C\,(2M)^{p-2}|Dw|^2 ,
\]
that is, the linearized operator is uniformly elliptic \emph{from below}; the whole difficulty comes
from the upper bound $|F|\le\check C\rho|Dw|$ with $\rho$ unbounded. One also has the H\"older-type
bound $|F|\le C_p|Dw|^{p-1}$, which is free of $\rho$; testing with $\eta^2w^{\beta}$ and using it
leads, after Young's inequality with exponents $\frac{2}{p-1},\frac{2}{3-p}$, to
\[
\int\eta^2w^{\beta-1}|Dw|^2\ \lesssim\ \int w^{\beta+\frac{p-1}{3-p}}\eta^{\frac{2(2-p)}{3-p}}|D\eta|^{\frac{2}{3-p}}
+\int\eta^2w^{\beta+1},
\]
in which no weight appears; unfortunately $\frac{p-1}{3-p}<1$ for $p<2$, so the powers of $w$ on the
two sides do not match and Moser's scheme degenerates into a Serrin-type Harnack inequality with an
additive constant, which is not sufficient for the strong comparison principle. Whether Theorem
\ref{thm:main4} holds for $1<p\le\frac32$ appears to be open; no counterexample is known, and none is
provided by the explicit solutions of Propositions \ref{prop:sharp} and \ref{prop:annulus}.
\end{remark}

\begin{remark}[Further extensions]
Beyond the applications collected in Section \ref{sec:applications}, the arguments of Sections
\ref{sec:morrey}--\ref{sec:sobolev} are purely local and use only the $C^{1,\tau}$ regularity of the
solution, the validity of the linearized equation \eqref{eq:linear-eq}, and a positive lower bound
for the source term on compact subsets. They therefore apply, with the same threshold
$p>\frac32$, to source terms $f=f(x,u)$ that have a sign and are locally Lipschitz in $u$
uniformly in $x$, in particular to H\'enon-type nonlinearities $f=|x|^{s}g(u)$, $s\ge0$, and,
in the vanishing-source setting of \cite{Sci2}, whenever the summability estimate
\cite[Theorem 3.1]{Sci2} is available: the only property of that estimate which we use is the
exponent $\sigma<p-1$, and the dichotomy of Step 2 in the proof of Theorem \ref{thm:morrey} is
insensitive to the structure of $f$ beyond \eqref{eq:cacc1}--\eqref{eq:cacc2}.
\end{remark}

%%%%%%%%%%%%%%%%%%%%%%%%%%%%%%%%%%%%%%%%%%%%%%%%%%%%%%%%%%%%%%%%%%%%%%%%%%%%%%%
\section*{Statements and Declarations}
%%%%%%%%%%%%%%%%%%%%%%%%%%%%%%%%%%%%%%%%%%%%%%%%%%%%%%%%%%%%%%%%%%%%%%%%%%%%%%%
%\noindent\textbf{Funding.} This research is funded by ...
\smallskip

\noindent\textbf{Conflict of interest.} The author declares that he has no
conflict of interest.

\smallskip

\noindent\textbf{Data availability.} Data sharing is not applicable to this
article, as no datasets were generated or analysed during the current study.

\smallskip

\noindent\textbf{Use of AI tools.} AI-assisted tools were used in developing and drafting this
work; the author has verified its contents and takes full responsibility.

\end{document}